\documentclass[11pt]{article}
\usepackage[utf8]{inputenc}

\usepackage[a4paper]{geometry}
\usepackage[Lenny]{fncychap}
\usepackage{mathtools}
\usepackage{amsmath,amsthm,amssymb}
\usepackage{amsxtra} 
\usepackage{amsfonts}
\usepackage{latexsym}
\usepackage{hhline}
\usepackage{stmaryrd}

\numberwithin{equation}{section}  

\usepackage{txfonts}
\usepackage{breakurl}
\usepackage{bbm}
\usepackage{mathrsfs}
\usepackage{xstring}
\usepackage{diagbox}

\usepackage{float}
\usepackage{enumitem}

\usepackage{palatino}
\usepackage{newcent}
\usepackage{mathptmx}
\usepackage[mathcal]{euscript}

\usepackage{tocloft} 
\usepackage{tocvsec2}
\usepackage{makeidx}

\usepackage{graphicx}
\usepackage{upref}
\usepackage{xcolor}
\usepackage[active]{srcltx}
\usepackage[colorlinks=true,linkcolor=blue,citecolor=magenta,urlcolor=red]{hyperref}
\usepackage[usenames,dvipsnames]{pstricks} 
\usepackage{epsfig} 
\usepackage{subcaption}
\usepackage{xypic}
\usepackage{tikz}
\usetikzlibrary{matrix,arrows,decorations.pathmorphing,calc,shapes}

\usepackage{tcolorbox}

\definecolor{verdeoscuro}{rgb}{0,0.3,0.2}
\definecolor{verdeclaro}{rgb}{0.93,1,0.9}
\definecolor{azulclaro}{rgb}{0.01,0.5,0.8}
\definecolor{gris}{rgb}{0.5,0.5,0.7}
\definecolor{celeste}{rgb}{0.02,0.7,0.9}
\definecolor{marron}{rgb}{0.3,0.1,0.0}

\def\tcv{\textcolor{violet}}

\def\tcb{\textcolor{blue}}

\theoremstyle{definition}

\newtheorem{theorem}{Theorem}[section]
\newtheorem{prop}[theorem]{Proposition}

\newtheorem{Lem}[theorem]{Lemma}
\newtheorem{lem}[theorem]{Lemma}
\newtheorem{Cor}[theorem]{Corollary}
\newtheorem{coro}[theorem]{Corollary}
\newtheorem{Teor}[theorem]{Theorem}
\newtheorem{teo}[theorem]{Theorem}
\newtheorem{Def}{Definition}[section]
\newtheorem{defi}{Definition}[section]
\newtheorem{Rem}{Remark}
\newtheorem{rem}{Remark}
\newtheorem{nota}{Notation}
\newtheorem{ex}[theorem]{Example}

\newenvironment{Proof}{\noindent{\it Proof.}}{\hfill$\Diamond$}

\DeclareMathAlphabet{\pazocal}{OMS}{zplm}{m}{n}

\newcommand{\Z}{\mathbb{Z}}

\newcommand{\C}{\mathbb{C}}

\def\b{\mathtt{b}}

\def\d{\mathtt{d}}

\def\Innder{\mathrm{Innder}}

\def\k{\kappa}

\def\To{\Rightarrow}
\def\ad{\mathrm{ad}}

\def\ot{\otimes}
\def\wt{\widetilde}

\newcommand{\Traza}{\operatorname{tr}}

\newcommand{\killing}{\kappa}
\newcommand{\End}{\operatorname{End}}

\newcommand{\sJac}{\operatorname{sJac}}

\newcommand{\Tag}{\operatorname{TAG}}
\newcommand{\Tkk}{\operatorname{TKK}}
\newcommand{\T}{\operatorname{T}}

\def\InnDer{\mathrm{InnDer}}

\DeclareMathOperator{\Img}{Im}
\DeclareMathOperator{\Ker}{Ker}

\DeclareMathOperator{\Hom}{Hom}

\DeclareMathOperator{\Der}{Der}

\newcommand{\act}[1]{(#1)}

\newcommand{\hdot}{{\raise1.2pt\hbox to0.3em{\Large $\dot$}}} 
\newcommand{\id}{\mathrm{id}}

\def\Mod{\ \mathrm{Mod}\ }

\def\id{\mathrm{Id}}
\def\Id{\mathrm{Id}}
\def\sl{\mathfrak{sl}}
\def\gl{\mathfrak{gl}}

\newcommand{\g}{\mathfrak{g}}

\newcommand{\h}{\mathfrak{h}}

\newcommand{\A}{\mathfrak{A}}
\newcommand{\B}{\mathfrak{B}}

\newcommand{\osp}{\mathfrak{osp}}

\newcommand{\Sl}{\mathfrak{sl}}

\newcommand{\SL}{\mathrm{SL}}

\newcommand{\D}{\mathfrak{D}}
\newcommand{\J}{\mathfrak{J}}

\newcommand{\cJ}{\mathcal{J}}
\newcommand{\M}{\mathcal{M}}

\newcommand{\ShL}{\sl_2{\text{-\bf SS}}}
\newcommand{\VShL}{\sl_2{\text{-\bf VSS}}}
\newcommand{\ShsL}{\sl_2{\text{-\bf supSS}}}
\newcommand{\VShsL}{\sl_2{\text{-\bf supVSS}}}

\newcommand{\Jsa}{{\text{\bf SJor}}}

\newcommand{\JterS}{{\text{\bf J-STer}}}
\def\rj{\rangle_{\!{}_\J}}
\def\rg{\rangle_{\!\!{}_\g}}
\def\rjm{\rangle _{\!\!{}_{\J,\M}}}
\def\pjm{\partial^{{}^{\J,\M}}}
\def\pg{\partial^{{}^{\g}}}

\makeindex  
\makeatletter

\author{Gonzalo Gutierrez and Marco Farinati}
\author{Marco A. Farinati
\thanks{Dpto de Matem\'atica FCEyN UBA - IMAS (Conicet). 
e-mail: mfarinat@dm.uba.ar.
} \ and
Gonzalo Gutierrez
\thanks{Dpto de Matem\'atica FCEyN UBA. 
e-mail: gonzalounsafcex@gmail.com 
Both authors are
partially supported by 
UBACyT and
 PICT 2022-05-00026:
Grupos cu\'anticos, categor\'ias trenzadas e 
invariantes de nudos.}
 }

\title{The Tits construction for \textit{short} $\sl_2$-super-structures}

\begin{document}

\maketitle

\begin{abstract}

    In this paper, we generalize the Tits construction for Lie superalgebras such that $\sl_2$ acts by even derivations and decompose, as $\sl_2$-module, into a direct sum of copies of the adjoint, the natural and  the trivial representations. 
    This construction generalizes the one provided by Elduque et al in \cite{EBCC23}, and it is possible to described the $\sl_2$-Lie superstructure in terms of $\J$-ternary superalgebras as a super version of the defined by Allison. We extend the Tits-Kantor-Koecher construction and the Tits-Allison-Gao functor that define a short $\sl_2$-Lie superalgebra from a $\J$-ternary superalgebra $(\J,\M)$. 
    Our setting includes and generalizes both results of \cite{EBCC23} and \cite{S22}.
    
\end{abstract}

{\small
\tableofcontents
}

\section*{Introduction}\label{Sec: Introd}

Consider the Lie algebra $\sl_2=\sl_2(\C)$ with the standard basis $e=
\big(\begin{smallmatrix}
0&1\\
0&0
\end{smallmatrix}\big)$, $f=
\big(\begin{smallmatrix}
0&0\\
1&0
\end{smallmatrix}\big)$, $h=
\big(\begin{smallmatrix}
1&0\\
0&-1
\end{smallmatrix}\big)$ and brackets $[e,f] =h$, $[h,e] = 2e$, $[h,f] = -2f$ and let
$
V=\C e_1 \oplus \C e_2 \cong \C^2
$
be the irreducible 2-dimensional representation of $\sl_2$.
If $\g$ is a Lie algebra that admits an action of $\sl_2$ by derivations, then one can decompose
$\g$ as an $\sl_2$-module
\[
\g\cong \bigoplus_{n\geq 0} V_{(n)}\ot \C^{m_n}
\]
and the bracket of $\g$
\[
[-,-]:\g\ot\g\to\g
\]
will be a map of $\sl_2$-modules.
This simple observation impose serious restrictions on the possible $\sl_2$-linear brackets
that makes one think that a possible characterization of $\sl_2$-Lie algebras is reasonable.
A very convenient and classical situation is when 
the $\sl_2$-module structure is of the form
\begin{align}
    \g \cong \sl_2\ot J \ \oplus\  D
\end{align}
where $J$ is just a multiplicity vector space and $D=\g^{\sl_2}$. That is, the only $\sl_2$-representations
that appears in $\g$ as $\sl_2$ module is the adjoint $\sl_2=\sl_2^{\ad}=V_{(2)}$ and the trivial one $V_{(0)}$. If this is the case,
we call $\g$ a {\bf very short} $\sl_2$-algebra.

A classical result of \cite{Tits} known as the \textit{``Tits construction''} shows
that if a Lie algebra $\g$ is very short, then it determines a structure of a Jordan algebra on $J$ 
such that $D$ acts by Jordan-derivations on $J$.
This Jordanian properties also determines the Lie structure on $\g = \sl_2\ot J\oplus \Innder(J)$, where $\Innder(J)$
is the Lie algebra of "inner derivations" of $J$, and 
Tits uses this construction to produce exceptional Lie algebras from some exotic Jordan ones.

A generalization to the super case is clear. A super Lie algebra $\g $ is called
$\sl_2$-superalgebra if $\sl_2$ acts by even derivations of $\wt{\g}$, and
we say that an $\sl_2$-Lie {\bf superalgebra} $\g $ is  \textbf{very short} if, as $\sl_2$-module,
\begin{align}\label{eq: very short sl2-Lie-super}
    \g  \cong \sl_2\ot \J  \ \oplus\ \D
\end{align}
where $\J = J_0 \oplus J_1$ is a super vector space with trivial $\sl_2$ action
 and $\D = D_0 \oplus D_1=\g ^{\sl_2}$.

In \cite{BC18}, \cite{Ka19} and \cite{S22}, this (super) situation is considered;  they
mention two types of constructions, knowing in the  
literature as the $\Tkk$-construction and the $\Tag$-construction, which define two {\em very short} Lie superalgebras $\Tkk(\J)$ and $\Tag(\J)$ from a Jordan superalgebra $\J$

\[\Tkk(\J):=\sl_2\ot \J\oplus\Innder(\J),\ \ 
\Tag(\J):=\sl_2\ot \J\oplus \B^s(\J)
\]
In the $\Tkk$ case,
the brackets are defined in terms of the Jordan multiplication, the Killing form, 
and the natural action of $\Innder(\J)$ by derivations on $\J$, generalizing to the super case the
classical Tits construction. In case of $\Tag(\J)$, a (super) Lie algebra $\B^s(\J)$ is introduce
in order to give a functorial replacement of $\Innder(\J)$, the name $\Tag$ comes from the
"Tits-Allison-Gao" functor. This $\Tag$ functor is adjoint to the Tits functor $\T$, that assign  a (super) Jordan algebra from a very short (super) $\sl_2$-algebra (see \cite{Ka19,S22}).

Notice that none of the above constructions involve the natural $2$-dimensional $\sl_2$-representation $V$.
In \cite{EBCC23,Sta22}  two types of (usual, non super) 
$\sl_2$-Lie algebras are mentioned,  named \textbf{short} and 
\textbf{very short} structures. 
An $\sl_2$-Lie algebra $\g$ is called {\bf short} if $\g$ decomposes, as $\sl_2$-module, into a direct sum of copies of the adjoint, the natural and  the trivial representations. That is, there is a $\sl_2$-module isomorphism:
\begin{align}\label{eq: short sl2-Lie}
    \g \cong  \sl_2\ot J \ \oplus \  V\ot M \ \oplus\ D
\end{align}
where $J$, $M$ and $D$ are vector spaces with trivial structure as $\sl_2$-module and  $D=\g^{\sl_2}$. 
The {\em very short} case coincides with the short case exactly when  with $M=0$.

\begin{ex}
A minimal but interesting example is the 3-dimensional Heisenberg Lie algebra 
$\h_3=\C e_1\oplus\C e_2\oplus \C z$ with bracket $[e_1,e_2]=z$, that  is a  short algebra, since
\[
\h_3\cong V\oplus \C z
\]
but it is not very short.
\end{ex}

In \cite{EBCC23} (see also \cite{EO11}), the authors provide
a Tits construction version for short Lie algebras and, in this case, the conditions on $(J,M)$ 
can be given in terms of $J$-ternary algebras defined by Allison \cite{All76, Hein83}: $J$ is a 
Jordan algebra, $M$ is a Jordan module and there exist a ternary operation on $M$ that satisfies
certain axioms (see Definition \ref{Def: super J-ternary}). This defines a $\Tkk$-construction 
for the non-super case, but they not define a $\Tag$-construction for short algebras,
the notion of $\B^s(J,M)$ 
is not mentioned in those papers.

In this work we extend the definition of short $\sl_2$-algebras to the super case, providing a generalization of the Tits construction and of the adjoint pair of functors $(\Tag,\T)$ to both the short and super cases.
\begin{defi}
    We say that a $\sl_2$-Lie {\bf superalgebra} $\g $ is  \textbf{short} if there exist an $\sl_2$-module
   decomposition:
\begin{align}\label{eq: short sl2-Lie-super}
    \g  = \sl_2\ot \J \ \oplus\  V\ot \M \ \oplus\ \D
\end{align}
where $\J = J_0 \oplus J_1$, $\M = M_0 \oplus M_1$ and $\D = D_0 \oplus D_1$ are  super vector spaces with trivial
structure as $\sl_2$-module; necessarily $\D = \g ^{\sl_2}$.
If $\g $ is short with $\M=0$, then we say $\g $ is {\bf very short}.
\end{defi}
First, we prove the super version of the algebraic characterization of $(\J,\M)$. That is, the bracket of $\g $ determines a super Jordan structure on $\J$, and $\M$ is a special Jordan supermodule endowed with a (super)ternary operation that satisfying some axioms (see Definition \ref{Def: super J-ternary}).
Secondly, we give a functorial construction replacement of the "inner
derivations generated by $\J$ and $\M$", that we call $\B^s(\J,\M)$, generalizing $\Tag$
for short superalgebras (including short non superalgebras). On the way of looking for a
functorial Tits construction, adapting ideas in \cite{S22} from the very short case
to the short case,
we give
explicitly all Jordan operations in terms of the Lie bracket of $\g$ 
and the $\sl_2$-module structure.

The following are names used in the literature and through this work:
\begin{itemize}
    \item $\ShL$ = the category of Short $\sl_2$-Lie algebras.
    \item $\VShL$ = the category of Very Short $\sl_2$-Lie algebras.
    \item $\ShsL$ = the category of Short Super $\sl_2$-Lie algebras.
    \item $\VShsL$ = the category of Very Short Super $\sl_2$-Lie algebras.
\end{itemize}

This work is organized as follows:

In Section \ref{Sec: Preliminaries} we introduce the basic definitions and notations for 
Lie and Jordan superalgebras that are needed here. 
In Section \ref{Sec: short sl2-super-Lie} we consider  short $\sl_2$-Lie 
{\bf superalgebras} and we give the conditions on the pair $(\J,\M)$ from the Lie superalgebra structure of $\g $,
that are described by the Theorem \ref{Teo: STits-1}, generalizing \cite[Thm 2.4]{EBCC23}.
For $\J = J_0$ and $\M =0$ as vector spaces we obtain the classical Tits construction. For $\J = J_0$, 
$\M = M_0$ we obtain the construction of \cite{EBCC23}. The case of \cite{S22} corresponds to $\M =0$. Hence, 
the decomposition \eqref{eq: short sl2-Lie-super} covers all the cases previously mentioned.

In \ref{Subsec: J-ternary super} we extend the concept of $\J$-ternary superalgebra $(\J,\M)$ to the super case
and, from Theorem \ref{Teor: (sJ,sM) as sJ-tern}, we show that every  short $\sl_2$-super Lie algebra 
defines a $\J$-ternary superalgebra structure on $(\J,\M)$. 
Conversely, from  Theorem \ref{Teor: sJ-tern to til(g)} we show that every $\J$-ternary superalgebra $(\J,\M)$ defines 
a short $\sl_2$-super Lie algebra, given a $\Tkk$ ``version'' for ternary cases. Thus, from both constructions we have the following diagram:
\begin{equation*}
    \xymatrix{
	\g  \ar@{~>}[rr]^{\text{Theorem \ref{Teor: (sJ,sM) as sJ-tern}}} & & (\J,\M) 
 \ar@{~>}[rr]^{\text{Theorem \ref{Teor: sJ-tern to til(g)}}} & &	\g(\J,\M)\\	
    }
\end{equation*}
where $\g(\J,\M) = \Tkk(\J,\M)$.

In Section \ref{Sec: Tag(J,M)-T(g)}, we define the Lie superalgebra $\B ^s(\J,\M)$ that plays the role of a functorial
replacement of $\InnDer(\J,\M)$, and we define
$\Tag(\J,\M)$ as a generalization of 
the $\Tag(\J)$ construction, but for $\J$-ternary superalgebras $(\J,\M)$. First, in \ref{Subsec: Tag(J)} we recall 
the $\Tag$-construction for a Jordan superalgebra $\J$ that defines the functor $\Tag$ from the category of Jordan
superalgebras $\Jsa$ to the category $\VShsL$. 
Following \cite{S22}, we show that $\Tag(\J,\M)$ is functorial and $\Tag:\JterS \to \ShsL$ defines a functor between the category of $\J$-ternary superalgebras $\JterS$ and the category $\ShsL$. 

In \ref{Subsec: funtor hat(T)} we consider the natural extension of the functor $\T(\g)=(J(\g),\cdot)$ 
from $\sl_2$-{\bf very} short algebras into
Jordan algebras
to a functor $\hat{\T}(\g)=(\J(\g),\M(\g),\cdot,\bullet,\star,(-,-,-))$
 from short $\sl_2$-super short algebras into super Jordan ternary algebras and we prove 
in Theorem \ref{Teo: adj T-Tag} that $\Tag$ is left adjoint to
$\hat{\T}$.

\section{Preliminaries}\label{Sec: Preliminaries}

\begin{Def}
A $\sl_2$-Lie algebra is a Lie algebra $\g$ together with an action of $\sl_2$ by derivations. 
That is, $\rho:\sl(2,\C)\to\Der(\g)$ is a morphism of Lie algebras.\end{Def}

We denote by $\killing(x,y) = 4\Traza(xy)$
the Killing form of $\sl_2$ and $\det:V\ot V\to \C$ the determinant form on $V$.
A superalgebra $\A$ is a $\Z_2$-graded space $\A = A_0 \oplus A_1$ with a product $*:\A \ot \A \to \A$ verifying
$A_i* A_j \subset A_{i+j}$ for $i,j\in \Z_2$. The component $A_0$ is called the even part of $\A$ and the 
component $A_1$ is called the odd part of $\A$ respectively.

\begin{Def}\label{Def: Super-Lie}
A {\bf Lie superalgebra} is a super-spaces $\g=\g_0\oplus\g_1$ with a product $[-,-]:\g\ot\g \to \g$ satisfying
\begin{enumerate}[label=\roman*]
    \item $[x,y] =- (-1)^{|x||y|}[y,x]$,

    \item $[[x,y],z] +(-1)^{|x|(|y|+|z|)}[[y,z],x] +(-1)^{|z|(|x|+|y|)}[[z,x],y] = 0$.
\end{enumerate}
An {\bf $\sl_2$-Lie superalgebra} is a Lie superalgebra $\g=\g_0\oplus\g_1$ such that $\sl_2$ acts on $\g$ by even derivations.
\end{Def}

Recall $V=V_1=\C^2$, the representation defining  $\sl_2$, that is the one
with basis
$e_1=(1,0)$, $e_2=(0,1)$, where the action of $\sl_2$ is given by
$e=
\big(\begin{smallmatrix}
0&1\\
0&0
\end{smallmatrix}\big)$, $f=
\big(\begin{smallmatrix}
0&0\\
1&0
\end{smallmatrix}\big)$, $h=
\big(\begin{smallmatrix}
1&0\\
0&-1
\end{smallmatrix}\big)$.

\begin{Def}\label{Def: Symetriz}
We define the symmetrizer map $\odot:V_1\ot V_1 \to \Sl_2$ as the following bilinear map:
\begin{equation}\label{eq: Symmetrizer}
\begin{aligned}
    e_1\ot e_1 \to e, &\qquad e_1\ot e_2 \to -\tfrac{1}{2}h,\\
    e_2\ot e_2 \to -f,&\qquad e_2\ot e_1 \to -\tfrac{1}{2}h,\\
\end{aligned}
\end{equation}
\end{Def}


For the definition of a Jordan superalgebra, we will adopt the one proposed by \cite{PS19}:

\begin{Def}\label{Def: superJordan}
A Jordan superalgebra is a $\Z_2$-graded vector space $\J=J_0\oplus J_1$ with an operation $\cdot: \J \ot \J \to \J$ satisfying the following conditions on homogeneous elements:
\begin{enumerate}
    \item (Super Commutativity) $a\cdot b = (-1)^{|a||b|}b\cdot a$,

    \item (4-Multilinear Jordan Identity) 
    
    $(a\cdot b)\cdot (c\cdot d) + (-1)^{|b||c|}(a\cdot c)\cdot (b\cdot d) + (-1)^{|b||d|+|c||d|}(a\cdot d)\cdot (b\cdot c) =$\hskip 1cm

    \hskip 1cm $=((a\cdot b)\cdot c)\cdot d + (-1)^{|b||c|+|b||d|+|c||d|}((a\cdot d)\cdot c)\cdot b + (-1)^{|a||b|+|a||c|+|a||d|+|c||d|}((b\cdot d)\cdot c)\cdot a$.
\end{enumerate}
for all $a,b,c,d \in \J$ homogeneous.
\end{Def}

About the Jordan multilinear identity, we mention the following
results:
\rem{lem}\label{lema1}
Let $\cdot:\J\ot\J\to\J$ 
be a graded supercommutative product. For $a,b\in \J$
define the linear map
 $\langle-,-\rj :\J \otimes \J \to \End(\J)$ by:
\begin{align}\label{eq: <a,b>}
\langle a,b \rj  (c):=a\cdot(b\cdot c)-(-1)^{|a||b|}b\cdot(a\cdot c)
\end{align}
\begin{enumerate}
\item If $\J$ is a Jordan algebra then
$\langle a,b \rj $ derives the product, that is
\[
\langle a,b\rj (c\cdot d)
=
 (\langle a,b\rj  (c))\cdot d
   + (-1)^{|c|(|a|+|b|)}c\cdot (\langle a,b\rj (d) ),
\]
\item Using supercommutativity, 
the identity 
\begin{align}
\label{apuntobc}
\langle a\cdot b,c\rj  
+ (-1)^{|c|(|b|+|a|)}\langle c\cdot a,b\rj  + (-1)^{|a|(|c|+|b|)}\langle b\cdot c,a\rj  = 0
\end{align}
evaluated at $d$, agree with the Jordan identity
(not for (a,b,c,d) but for (a,b,d,c)). Hence,
$(\J,\cdot)$ is a super Jordan algebra if and only
 if the product
super commutative and  \eqref{apuntobc} holds.

\end{enumerate}

\begin{proof} Assertion 2. is straightforward. Assertion 1. y
probably well-known to specialists, we include a proof for convenience of the reader:

Consider the Jordan identity for $a,b,c,d$:

\[
\tcb{(a\cdot b)\cdot (c\cdot d)} 
+ (-1)^{|b||c|}(a\cdot c)\cdot (b\cdot d) 
+ (-1)^{|b||d|+|c||d|}\tcv{(a\cdot d)\cdot (b\cdot c)} 
=\hskip 1cm
\]
\[
\hskip 1cm
=((a\cdot b)\cdot c)\cdot d + (-1)^{|b||c|+|b||d|+|c||d|}((a\cdot d)\cdot c)\cdot b 
+ (-1)^{|a||b|+|a||c|+|a||d|+|c||d|}((b\cdot d)\cdot c)\cdot a
\]
and the Jordan identity for $d,a,b,c$:
\[
\tcv{(d\cdot a)\cdot (b\cdot c)}
 + (-1)^{|a||b|}(d\cdot b)\cdot (a\cdot c) 
 + (-1)^{|a||c|+|b||c|}\tcb{(d\cdot c)\cdot (a\cdot b)} =
\hskip 1cm
\]
\[
\hskip 1cm 
=((d\cdot a)\cdot b)\cdot c + (-1)^{|a||b|+|a||c|+|b||c|}((d\cdot c)\cdot b)\cdot a 
+ (-1)^{|d||a|+|d||b|+|d||c|+|b||c|}((a\cdot c)\cdot b)\cdot d
\]
Rearranging this second identity using super commutativity
we get
\[
 (-1)^{|a||d|+|b||d| +|c||d|}\tcb{(a\cdot b)\cdot (c\cdot d)}
+ (-1)^{|b||d| +|a||d|+|b||c|+|c||d| }(a\cdot c)\cdot (b\cdot d)
+(-1) ^{|a||d|}  \tcv{(a\cdot d)\cdot (b\cdot c) }
=
\hskip 1cm
\]
\[
\hskip 1cm 
=   (-1)^{|a||d|+|b||a|+|b||d||+|c||b|+|c||a|+|c||d|}
c\cdot (b\cdot (a\cdot d)) 
   + (-1)^{|a||d|+|b||d|+|d||c|}a\cdot (b\cdot (c\cdot d)) 
\]
\[+ (-1)^{|a||d|+|d||b|+|c||d|+|a||b|+|a||c|}(b\cdot (c\cdot a))\cdot d
\]
Multiplying by  $(-1)^{|a||d|+|b||d| +|c||d|}$ we get 
an obvious equality between the respective LHS, and we deduce an equality
 of the corresponding RHS. This new equality, after
  multiplying by 
$(-1)^{|a||b|+|a||c|}$ is
\[
(-1)^{|b||c|}(c\cdot (b\cdot a))\cdot d 
+ b\cdot (c\cdot (a\cdot d)) 
+ (-1)^{|a||b|+|a||c|+|b||c|}a\cdot (c\cdot (b\cdot d))
   =\]
   \[=
   (-1)^{|c||b|}c  \cdot(b\cdot (a\cdot d)) 
   + (-1)^{|a||b|+|a||c|}a\cdot (b\cdot (c\cdot d)) 
+ (b\cdot (c\cdot a))\cdot d\]
which is equivalent to
\[
 b\cdot (c\cdot (a\cdot d)) 
 -   (-1)^{|c||b|}c\cdot (b\cdot (a\cdot d)) 
=\]
\[
=
 (b\cdot (c\cdot a))\cdot d
-(-1)^{|b||c|}(c\cdot (b\cdot a))\cdot d 
   + (-1)^{|a||b|+|a||c|}a\cdot (b\cdot (c\cdot d)) 
- (-1)^{|a||b|+|a||c|+|b||c|}a\cdot (c\cdot (b\cdot d))
\]
hence
\[
\langle b,c\rj (a\cdot d)
=
 (\langle b,c\rj  (a))\cdot d
   + (-1)^{|a|(|b|+|c|)}a\cdot (\langle b,c\rj (d) ).
\]
\end{proof}

\begin{rem}
Let  $(\J,\cdot)$ be a super Jordan algebra, $a,b\in\J$ 
homogeneous and $\langle a,b\rj :\J\to\J $ as before.
If $D:\J\to\J$ is another super derivation,  a straight
computation shows that
\[
[D,\langle a,b\rj ]:=D\circ \langle a,b\rj -
(-1)^{|a||D|}\langle a,b\rj \circ D
\]
\[
=
\langle D(a),b\rj +
(-1) ^{|a||D|}\langle D(a),b\rj 
\]
In particular, 
the subspace of $\End(\J)$ spanned by $\langle a,b \rj $, with $a,b\in\J$ homogeneous
is a Lie (sub)algebra,
 that we call, by definition, the Lie algebra of Inner derivations of $\J$.
We denote it by
 $\Innder(\J)$.
 \end{rem}

\begin{Def}\label{Def: J-super-mod}
    A special $\J$-supermodule is a vector super-space $\M$ with an action $\bullet:\J\ot \M \to \M$ such that $a\in \J \mapsto a\bullet-\in\End(\M)^{+}$ is a morphism of Jordan superalgebras, where $\End(\M)^{+}$ is the Jordan superalgebra with product
    \[
    S\cdot T = \tfrac{1}{2}(S\circ T + (-1)^{|S||T|}T\circ S).
    \]
\end{Def}

\section{Short $\sl_2$-Lie superalgebras}\label{Sec: short sl2-super-Lie}

Suppose that $\g  $ is a short $\sl_2$-Lie superalgebra with decomposition:
\begin{align}\label{eq: sl2-corta tipo 1}
    \g  = \sl_2\ot \J \; \oplus \; V\ot \M \; \oplus \; \D
\end{align}
where here $\D=\g ^{\sl_2}$ and we consider $\sl_2$ 
 and $V$ of even degree. 
 
With the following conventions on letters:
$\d\in D$, $x,y\in \sl_2$, $v,w\in V$, $a,b\in\J$, $m,n\in\M$, where they are 
also assumed to be homogeneous of degree $|\d|$, $|a|$, $|b|$, etc. The $\sl_2$-invariance implies that the bracket in $\g $ is given by the following description,

\begin{enumerate}[label=\roman*]
    \item $[\d,x\ot a] = x\ot\d\act{a}$, \label{eq: Brk-g-1}
    \item $[\d,v\ot m] = v\ot\d\act{m}$, \label{eq: Brk-g-2}
    \item $[x\ot a,y\ot b] = [x,y]\ot a\cdot b + \tfrac{1}{2}\killing(x,y)\langle a,b\rg $, \label{eq: Brk-g-3}
    \item $[x\ot a,v\ot m] =[x,v]\ot a\bullet m$, \label{eq: Brk-g-4}
    \item $[v\ot m,w\ot n] = v\odot w\ot m\star n 
    + \det(v,w)\pg_{m,n}$, \label{eq: Brk-g-5}
\end{enumerate}
for some operations $\cdot:\J\ot\J \to \J$, $\bullet:\J\ot \M \to \M$ and $\star:\M\ot \M \to \J$ 
homogeneous of degree zero and bilinear maps $\langle -,-\rg :\J\ot \J\to \D$, $\pg:\M\ot\M \to \D$ with degrees $|\langle a, b \rg | = |a| + |b|$ and $|\pg_{m, n}| = |m| + |n|$.

\subsection{Main properties of the operations $\cdot,\bullet,\star,
\langle-,-\rg $ and $\pg$.}

The super-anticommutativity of the bracket of $\g $ implies that:
\begin{enumerate}[label=\roman*]
    \item $a \cdot b = (-1)^{|a||b|} b \cdot a$, \label{eq: sconm-cdot JM}
    \item $\langle a, b \rg  = -(-1)^{|a||b|} \langle b, a \rg $, \label{eq: santisim-<,> JM}
    \item $a \bullet m = (-1)^{|a||m|} m \bullet a$, \label{eq: sconm-bull JM}
    \item $m \star n = -(-1)^{|m||n|} n \star m$, \label{eq: sconm-star JM}
    \item $\pg_{m,n} = -(-1)^{|m||n|} \pg_{n, m}$, \label{eq: santisim-partial JM}
\end{enumerate}
Therefore, the maps $\cdot,\bullet,\pg$ are super-commutative and the maps $\star,\langle-,-\rangle$ are super-anticommutative. Let
\begin{align*}
    \sJac(\alpha,\beta,\gamma) = [\alpha,[\beta,\gamma]] - [[\alpha,\beta],\gamma] - (-1)^{|\alpha||\beta|}[\beta,[\alpha,\gamma]]
\end{align*}
be the Jacobi super-identity in $\alpha,\beta,\gamma \in \g $. This identity shows the conditions to be satisfied by these maps. In fact:
\begin{itemize}
    \item From $\sJac(\d_1,\d_2,-)=0$ we obtain that:
    \begin{align*}
        [\d_1,\d_2](\eta) = \d_1(\d_2(\eta)) - (-1)^{|\d_1||\d_2|}\d_2(\d_1(\eta)) 
    \end{align*}
    for all $\eta\in \J\oplus \M$.

    \item From $\sJac(\d,-,-)=0$ we obtain:
    \begin{align}
        \d\act{a\cdot b} &= \d\act{a}\cdot b + (-1)^{|\d||a|}a\cdot \d\act{b}\label{derivation1}\\
        \d\act{a\bullet m} &= \d\act{a}\bullet m + (-1)^{|\d||a|}a\bullet \d\act{m}\label{derivation2}\\
        \d\act{m\star n} &= \d\act{m}\star n + (-1)^{|\d||m|}m\star\d\act{n}\label{derivation3}\\
        [\d,\langle a,b\rg ] &= \langle \d\act a,b\rg  + (-1)^{|\d||a|}\langle a,\d\act b\rg \label{derivation4}\\
        [\d,\pg_{m,n}] &= \pg_{\d\act m,n} + (-1)^{|\d||m|}\pg_{m,\d\act n}\label{derivation5}
    \end{align}
    
    \item From $\sJac(x\ot a,y\ot b,z\ot c) = 0$ we obtain that:
    \begin{align*}
        [x,[y,z]]\ot &a\cdot(b\cdot c) + \tfrac{1}{2}\killing(x,[y,z])\langle a,b\cdot c\rg -[[x,y],z]\ot (a\cdot b)\cdot c-\tfrac{1}{2}\killing([x,y],z)\langle a\cdot b,c\rg \\
        & -(-1)^{|a||b|}[y,[x,z]]\ot b\cdot(a\cdot c)-(-1)^{|a||b|}\tfrac{1}{2}\killing(y,[x,z])\langle b,a\cdot c\rg \\
        & + (-1)^{|a|(|b|+|c|)+1}\tfrac{1}{2}\killing(y,z)x\ot \langle b,c\rg \act{a} -\tfrac{1}{2}\killing(x,y)z\ot \langle a,b\rg \act{c}\\
        & + (-1)^{|a||c|}\tfrac{1}{2}\killing(x,z)y\ot \langle a,c\rg \act{c} = 0
    \end{align*}
    Since the Killing  satisfies
    \[
    [x,[y,z]] = \tfrac{1}{2}\killing(x,y)z - \tfrac{1}{2}\killing(x,z)y,
    \]
    \[
    \killing([x,y],z) = \killing(x,[y,z]),
    \]
    then the equation above gives:
    \begin{align}
        \langle a,b\rg  \act c = a\cdot(b\cdot c) - (-1)^{|a||b|}b\cdot(a\cdot c)
    \end{align}
    \begin{align}
        \langle a\cdot b,c\rg  + (-1)^{|c|(|b|+|a|)}\langle c\cdot a,b\rg  + (-1)^{|a|(|c|+|b|)}\langle b\cdot c,a\rg  = 0
    \end{align}
    
    \item From $\sJac(x\ot a,y\ot b,v\ot m)=0$ we obtain:
    \begin{align*}
        x(y(v))\ot a\bullet(b\bullet m) - x(y(v))\ot (a\cdot b)\bullet m &-\tfrac{1}{2}\killing(x,y)v\ot \langle a,b\rg \act{m}\\
        & - (-1)^{|a||b|}y(x(v))\ot b\bullet(a\bullet m) = 0
    \end{align*}
    On one side, if $x=y=h$ and $v=e_1$ we have
    \begin{align}
        4\langle a,b\rg  \act m = a\bullet(b\bullet m) - (-1)^{|a||b|}b\bullet(a\bullet m)
    \end{align}
    and on other side, if $x=e$, $y=f$ and $v=e_1$ we have
    \begin{align}
        2\langle a,b\rg \act{m} &= a\bullet(b\bullet m)-(a\cdot b)\bullet m
    \end{align}
    that implies:
    \begin{align}
        (a\cdot b)\bullet m &= \tfrac{1}{2}\big(a\bullet(b\bullet m) + (-1)^{|a||b|}b\bullet(a\bullet m)\big) 
    \end{align}

    \item From $\sJac(x\ot a,v\ot m,w\ot n) = 0$ we obtain:
    \begin{align*}
        [x,v\odot w]\ot a\cdot (m\star n) &+ (-1)^{|a|(|m|+|n|)+1}\det(v,w)x\ot \pg_{m,n}\act{a} \\
        & = [x,v]\odot w\ot(a\bullet m)\star n + (-1)^{|a||m|}v\odot [x,w]\ot m\star(a\bullet n)\qquad (*)
    \end{align*}
    \begin{align*}
       \lambda \k(x,v\odot w)\langle a,m\star n\rg  
= \mu \det(x(v),w)\pg_{a\bullet m,n}
+(-1)^{|a||m|}\det(v,x(w))\pg_{m,a\bullet n} \qquad (**)
    \end{align*}
    If we evaluate $(*)$ with $x=h$, $v=e_1$ and $w=e_2$ we obtain that:
    \begin{align}
        2\pg_{m,n}\act a = (-1)^{|a|(|m|+|n|)}(a\bullet m)\star n - (-1)^{|a||n|}m\star(a\bullet n)
    \end{align}
    but if we evaluate $(*)$ with $x=h$, $v=e_1$ and $w=e_1$ we obtain that:
    \begin{align}
        a\cdot(m\star n) = \tfrac{1}{2}\big((a\bullet m)\star n + (-1)^{|a||m|}m\star(a\bullet n)\big)
    \end{align}
    On other hand, $(**)$ with $x=f$, $v=e_1$ and $w=e_1$ implies:
    \begin{align}
        2\langle a,m\star n\rg  = -\pg_{a\bullet m,n}+(-1)^{|a||m|}\pg_{m,a\bullet n}
    \end{align}

    \item Finally, from $\sJac(u\ot m,v\ot n,w\ot r) = 0$ we obtain:
    \begin{align*}
        &(-1)^{|n|+|m|(|n|+|r|)}[v\odot w,u]\ot (n\star r)\bullet m + \det(v,w)(-1)^{|n|+|m|(|n|+|r|)}u\ot \pg_{n,r}(m)\\
        &+(-1)^{|n|}(u\odot v)(w)\ot (m\star n)\bullet r 
        + \det(u,v)(-1)^{|n|}w\ot \pg_{m,n}(r)\\
        &+(-1)^{|n|(|r|+1)}(u\odot w)(v)\ot (m\star r)\bullet n + \det(u,w)(-1)^{|n|(|r|+1)}v\ot \pg_{m,r}(n) = 0
    \end{align*}
    Therefore, if we consider $u= e_1$, $v= e_2$ and $w = e_1$ we have:
    \begin{align*}
        -\pg_{m,n}(r) - &(-1)^{|m|(|n|+|r|)}\pg_{n,r}(m) = \tfrac{1}{2}(m\star n)\bullet r\\
        &\qquad  - (-1)^{|n|(|r|+1)}(m\star r)\bullet n + \tfrac{1}{2}(-1)^{|m|(|n|+|r|)}(n\star r)\bullet m
    \end{align*}
    By the super-anticommutativity of $\pg_{n,r}$, we have the identity
    \begin{align}
        \pg_{m,n}\act r + (-1)^{|m|(|n|+|r|)+|n||r|}&\pg_{r,n}\act m = \\
        \tfrac{1}{2}(m \star n)\bullet r + \tfrac{1}{2}&(-1)^{|m|(|n|+|r|)+|n||r|}(r\star n)\bullet m - (-1)^{|n||r|}(m\star r)\bullet n \nonumber
    \end{align}
\end{itemize}

An immediate consequence is the following:

\begin{Teor}\label{Teo: STits-1}
    A vector superspace $\g $ with isotypic decomposition \eqref{eq: sl2-corta tipo 1} is a short $\sl_2$-Lie superalgebra if and only if there exist super-symmetric maps $\cdot:\J\ot\J \to \J$, $\bullet:\J\ot \M \to \M$, $\pg:\M\ot \M \to \D$ and super-antisymmetric maps $\langle -,-\rg :\J\ot \J \to \D$, $\star:\M\ot \M \to \M$ than describe the brackets of $\g $ by \eqref{eq: Brk-g-1}-\eqref{eq: Brk-g-5} and satisfy the following conditions: (the convention for the  letters is 
    $a,b,c\in\J$, $m,n,r\in\M$, $d\in \D$)
    \begin{enumerate}
        \item $(\J,\cdot)$ is a Jordan superalgebra.\label{itm: super-Jordan JM}

        \item $\M$ is a special Jordan $\J$-supermodule with action $a\otimes m \mapsto a\bullet m$. That is:
        \begin{align}\label{eq: (a.b)bul m}
        (a\cdot b)\bullet m = \frac{1}{2}\big(a\bullet(b\bullet m) + (-1)^{|a||b|}b\bullet(a\bullet m)\big)
        \end{align}

        \item $\D$ super-derives the operations $(\J\oplus \M,\cdot,\bullet,\star)$, that is:
        \begin{align}
            \d\act{a\cdot b} &= \d\act{a}\cdot b + (-1)^{|\d||a|}a\cdot \d\act{b} \label{eq: d(a dot b)}\\
            \d\act{a\bullet m} &= \d\act{a}\bullet m + (-1)^{|\d||a|}a\bullet \d\act{m} \label{eq: d(a bul m)}\\
            \d\act{m\star n} &= \d\act{m}\star n + (-1)^{|\d||m|}m\star\d\act{n} \label{eq: d(m star n)}
        \end{align}
        and for all $\d_1,\d_2\in \D$, $\eta\in \J\oplus \M$ holds:
        \begin{align}
            [\d_1,\d_2](\eta) = \d_1(\d_2(\eta)) - (-1)^{|\d_1||\d_2|}\d_2(\d_1(\eta))\label{eq: [d1,d2]}  
        \end{align}

        \item For all $a,b\in J$, $m,n\in \M$ and $\d\in \D$ the bilinear maps $\langle a,b\rg $ y $\pg_{m,n}$ satisfy:
        \begin{align}
            [\d,\langle a,b\rg ] &= \langle \d\act a,b\rg  + (-1)^{|\d||a|}\langle a,\d\act b\rg  \label{eq: [d,<a,b>]}\\
            [\d,\pg_{m,n}] &= \pg_{\d\act m,n} + (-1)^{|\d||m|}\pg_{m,\d\act n} \label{eq: [d,partial_{m,n}]}
        \end{align}
        \begin{align}
        \langle a,b\rg  \act c &= a\cdot(b\cdot c) - (-1)^{|a||b|}b\cdot(a\cdot c) \label{eq: <a,b>-1 JM}\\
        4\langle a,b\rg  \act m &= a\bullet(b\bullet m) - (-1)^{|a||b|}b\bullet(a\bullet m) \label{eq: <a,b>-2 JM}
        \end{align}
        \begin{align}\label{eq: <ab,c>}
            \langle a\cdot b,c\rg  + (-1)^{|c|(|b|+|a|)}\langle c\cdot a,b\rg  + (-1)^{|a|(|c|+|b|)}\langle b\cdot c,a\rg  = 0
        \end{align}
        \begin{align}
            2\langle a,m\star n\rg  &= -\pg_{a\bullet m,n}+(-1)^{|a||m|}\pg_{m,a\bullet n} \label{eq: angle-partial}\\
            2\pg_{m,n}\act a &= (-1)^{|a|(|m|+|n|)}(a\bullet m)\star n - (-1)^{|a||n|}m\star(a\bullet n) \label{eq: partial_{m,n}-1}
        \end{align}
        \begin{equation}\label{eq: partial_{m,n}-2}
        \begin{aligned}
            \pg_{m,n}\act r + (-1)^{|m|(|n|+|r|)+|n||r|}&\pg_{r,n}\act m = \\
            \tfrac{1}{2}(m \star n)\bullet r + \tfrac{1}{2}&(-1)^{|m|(|n|+|r|)+|n||r|}(r\star n)\bullet m - (-1)^{|n||r|}(m\star r)\bullet n
        \end{aligned}
        \end{equation}
        In particular, $\D$ must act by inner derivations on the Jordan superalgebra $(\J,\cdot)$.

        \item For all $a,b\in \J$ and $m,n\in \M$ hold the following compatibility rule:
        \begin{align}\label{eq: a.(m*n)}
        a\cdot(m\star n) = \frac{1}{2}\big((a\bullet m)\star n + (-1)^{|a||m|}m\star(a\bullet n)\big)
        \end{align}
    \end{enumerate}
\end{Teor}

\begin{Proof} From the discussion of this section and
 Lemma \ref{lema1} we see that the list of
 conditions are equivalent to superantisymmetry of the bracket and super Jacobi identity.
\end{Proof}

\begin{Rem} From \eqref{eq: [d,<a,b>]} and \eqref{eq: [d,partial_{m,n}]} we conclude:
    \begin{align}
 [\langle a,b\rg ,\langle c,d\rg ] &= 
 \langle \langle a,b\rg (c),d\rg  
 + (-1)^{(|a|+|b|)|c|}\langle c,\langle a,b\rg (d)\rg 
 \\
 [\langle a,b\rg ,\pg_{m,n}] 
 &   = \pg_{\langle a,b\rg (m),n}
    + (-1)^{(|a|+|b|)|m|}\pg_{m,\langle a,b\rg (n)}
    \\
 [\pg_{m,n},\langle a,b\rg ] 
 &= \langle \pg_{m,n}(a),b\rg  
 + (-1)^{(|m|+|n|)|a|}\langle a,\pg_{m,n}(b)\rg
  \\
&= -(-1)^{(|a|+|b|)(|m|+|n|)}
 [\langle a,b\rg, \pg_{m,n}]
 \\
&= -(-1)^{(|a|+|b|)(|m|+|n|)}
\pg_{\langle a,b\rg(m),n}+  
(-1)^{(|a|+|b|)|n|}
\pg_{m,\langle a,b\rg(n)}
\\
[\pg_{m,n},\pg_{r,s}]
&= \pg_{\pg_{m,n}(r),s} 
+ (-1)^{(|m|+|n|)|r|}\pg_{r,\pg_{m,n}(s)}
\label{ganchogancho}
    \end{align}
   
\end{Rem}

\begin{rem}\label{Prop: deriv-rel}
    Let $\g\in \ShsL$ be a short $\sl_2$-Lie superalgebra with isotypic decomposition \eqref{eq: sl2-corta tipo 1} and $\J$, $\M$ as Theorem \ref{Teo: STits-1}. Then, for all $\d\in\D$, $a,b\in\J$, $m,n\in\M$ and $\eta\in \J\oplus\M$, holds:
    \begin{align}
        \d(\langle a,b\rg(\eta)) &= \langle \d(a),b\rg(\eta) + (-1)^{|\d||a|}\langle a,\d(b)\rg(\eta) + (-1)^{|\d|(|a|+|b|}\langle a,b\rg(\d(\eta))\\
        \d(\partial_{m,n}(\eta)) &= \partial_{\d(m),n}(\eta) + (-1)^{|\d||m|}\partial_{ m,\d(n)}(\eta) + (-1)^{|\d|(|m|+|n|}\partial_{a,b}(\d(\eta)) \label{ganchogancho2}
    \end{align}
\end{rem}
\begin{Proof}
    In fact, from eq. \eqref{eq: [d1,d2]} and eq. \eqref{eq: [d,<a,b>]} we have:
    \begin{align*}
        \d(\langle a,b\rg(\eta)) - (-1)^{|\d|(|a|+|b|)}\langle a,b\rg\act{\d(\eta)} &= \langle \d(a),b\rg(\eta) + (-1)^{|\d||a|}\langle a,\d(b)\rg(\eta)
    \end{align*}
    and similarly, from eq. \eqref{eq: [d1,d2]} and eq. \eqref{eq: [d,partial_{m,n}]} we have:
    \begin{align*}
        \d(\partial_{m,n}(\eta)) - (-1)^{|\d|(|m|+|n|)}\partial_{m,n}(\d(\eta)) &= \partial_{\d(m),n}(\eta) + (-1)^{|\d||m|}\partial_{m,\d(n)}(\eta)
    \end{align*}
    Hence, both equalities are proven.
\end{Proof}

\subsection{$\sl_2$-Lie theoretical formulas for
 the operations $\cdot,\bullet,\star,\langle-,-\rg $ and $\pg$.}

\begin{rem}
Let $\g $ be an $\sl_2$-super Lie algebra with
 decomposition
\[
    \g  = \sl_2\ot \J \ \oplus\  V\ot \M \ \oplus\ \D
\]
Since the eigenvalues of $h$ in $V\ot\M$ are $\pm 1$, and the eigenvalues on $\sl_2\ot \J$ are $\pm2$ and $0$,
and {\em both} representations are generated under $\sl_2$
by vectors anihilated by $e$, one has
\[\M(\g)=
\{\alpha\in \g: h\alpha=\alpha\}= e_1\ot \M
\]
\[
\J(\g)
=\{\alpha\in\g: h\alpha=2\alpha\}=e\ot \J
\]
and 
$\D\cong \g ^{\sl_2}=\{\alpha\in\g: h\alpha=0=e\alpha\}$, 
\end{rem}

\begin{nota} If $a\in \J$ and $m\in\M$ we denote
\[
\wt a:=e\ot a,\hskip 1cm \wt m:=e_1\ot m\]
\end{nota}

%

From the form of the bracket on $\g$, and evaluating 
$x,y,v,w$ on special cases we get

\begin{enumerate}
    \item 
$
[x\ot a,y\ot b] = [x,y]\ot a\cdot b + \tfrac{1}{2}\killing(x,y)\langle a,b\rg \ \To$
\[
[h\ot a,h\ot b] = [h,h]\ot a\cdot b + \tfrac{1}{2}\killing(h,h)\langle a,b\rg
= \tfrac{1}{2}\killing(h,h)\langle a,b\rg
 \]
If we consider $\J(\g)=\{\alpha\in\g :h\alpha=2\alpha\}=e\ot\J$, and using the notation, for $a\in\J$,
\[
\wt a:=e\ot a\]
then, just using that $h\ot a=-[f,e]\ot a=-f( e\ot a)=-f(\wt a)$ we get
\[
\langle a,b\rg
=\frac{1}{\tfrac{1}{2}\killing(h,h)}[f\wt a,f\wt b]
=\frac{1}{4}[f\wt a,f\wt b]
    \]

Similarly
$[x\ot a,y\ot b] = [x,y]\ot a\cdot b + \tfrac{1}{2}\killing(x,y)\langle a,b\rg \ \To$
\[
[h\ot a,e\ot b] 
= [h,e]\ot a\cdot b + \tfrac{1}{2}\killing(h,e)\langle a,b\rg
= 2e\ot a\cdot b 
\]
\[
\To \wt {a\cdot b}=\frac12[f\wt a,\wt b] \]

    \item $[x\ot a,v\ot m] = x(v)\ot a\bullet m$,\; $[v\ot m,x\ot a] = -x(v)\ot m\bullet a$, 
implies

$[h\ot a,e_1\ot m] = h(e_1)\ot a\bullet m=e_1\ot a\bullet m$.
Now we use the characterization
\[
\M(\g)=\{\alpha\in\g: h\alpha=\alpha\}=e_1\ot\M\]
and the notation, for $m\in \M$,
\[
\wt m:=e_1 \ot m
\]
From the previous formula, together with $h=[e,f]=-[f,e]$,
we get
\[
\wt {a\bullet m}=-[f(e\ot a),(e_1\ot m)]=-[f\wt a,\wt m].
\]
Notice $h\wt a=2\wt a$ and $h\wt m=\wt m$ implies $h[\wt a,\wt m]=3[\wt a,\wt m]=0$,
because there is no egenvalue 3 in $\g $, so
\[
0=f0=f[\wt a,\wt m]=[f\wt a,\wt m]+[\wt a,f\wt m]
\ \To \ 
a\bullet m=-[f\wt a,\wt m]=[\wt a,f\wt m]
\]

    \item From $[v\ot m,w\ot n] =v\odot w\ot m\star n + \det(v,w)\pg_{m,n}$
    we get in particular
        
    $[e_1\ot m,e_1\ot n] =(e_1\odot e_1)\ot m\star n + \det(e_1,e_1)\pg_{m,n}=e\ot m\star n 
    \ \To \ 
   \wt{m\star n}=[\wt m,\wt n]$.
   
And finally, from
   
    $[e_1\ot m,e_2\ot n] =e_1\odot e_2\ot m\star n + \det(e_1,e_2)\pg_{m,n}$

    $[e_2\ot m,e_1\ot n] =e_2\odot e_1\ot m\star n + \det(e_2,e_1)\pg_{m,n}$

we get
$[e_1\ot m,e_2\ot n] 
    -[e_2\ot m,e_1\ot n]
    =$
    \[
    =\Big(e_1\odot e_2\ot m\star n + \det(e_1,e_2)\pg_{m,n}\Big)
-\Big(
   e_2\odot e_1\ot m\star n + \det(e_2,e_1)\pg_{m,n}\Big)
    \]

Using symmetry of the symmetric product and anti-symmetry of 
the determinant we conclude
\[
[e_1\ot m,e_2\ot n] 
    -[e_2\ot m,e_1\ot n]
        =2 \pg_{m,n}\]
    so
\[
\pg_{m,n}=\frac1{2}\Big([\wt m,f\wt n]-[f\wt m,\wt n]    \Big)
\]
    \end{enumerate}

We sumarize the following ``dictionary'':
\begin{prop}
\label{propdic1}
Let $\g=\sl_2\ot\J\oplus V\ot\M\oplus\D$ be an $\sl_2$-short super Lie algebra and identify
\[
\J(\g)=\{\alpha\in\g:h\alpha=2\alpha\}=e\ot\J\cong \J,
\hskip 1cm
\M(\g)=\{\alpha\in\g:h\alpha=\alpha\}=e_1\ot\M\cong \M
\]
\[
\hskip 2cm 
\wt a=e\ot a \leftrightarrow a
\hskip 5cm
\wt m=e_1\ot m \leftrightarrow m
\]
The operations $\cdot,\bullet,\star,
\pg,\langle -,-\rg$, in terms of the Lie bracket and $\sl_2$ action are given by:

\begin{align}
\wt a\cdot \wt b&=\frac12[f\wt a,\wt b] \label{cdotlie}\\
\wt {a\bullet m}&=-[f\wt a,\wt m]=[\wt a,f\wt m]\label{bulletlie}\\
   \wt { m\star n}&=[\wt m,\wt n]\label{starlie}\\
\wt {\langle a,b\rg}&=\frac{1}{4}[f\wt a,f\wt b]\label{langlelie}\\
\pg_{m,n}
=&\frac1{2}\Big([\wt m,f\wt n]-[f\wt m,\wt n] \Big)\label{partiallie}
    \end{align}

\end{prop}

\begin{rem}\label{rem5}
If $\g$ is a short $\sl_2$ algebra, $a,b\in\J(\g)$, then 
$\langle a,b\rg\in\D=\g^{\sl_2}\subset\sl_2$ is an element of $\g$,
while
$\langle a,b\rj\in\InnDer(\J)$ is a derivation of $\J$.  The relation between them 
is that $\D\subset\g$ acts on $\g$ by Lie
derivations that are $\sl_2$-linear. So, their elements
induce derivations on $\J$.
The relation between $\langle-,-\rg$ and
$\langle-,-\rj$ is simply
\[
[\langle a,b \rg,\wt  c]
=[\langle a,b \rg,e\ot c]
=e\ot\langle a,b \rj(c)
=\wt{\langle a,b \rj(c)}
\]
By abuse of notation we will sometimes write
$
[\langle a,b \rg,c]=\langle a,b \rj(c)$.
\end{rem}

\section{$\J$-ternary Superstructures}\label{Subsec: J-ternary super}

The following definition generalizes the concept of $J$-ternary algebra 
(see \cite{EO11,All76,Hein83}) for the super case:

\begin{Def}\label{Def: super J-ternary}
    Let $\J$ be a Jordan superalgebra with product $\cdot:\J\ot \J\to \J$, $\M$ a special $\J$-supermodule 
    with Jordan action $\bullet:\J\ot \M\to \M$ and super-anticommutative operation $\star: \M\ot \M \to \J$. Suppose that $(-,-,-):\M^{\ot 3} \to \M$ is a trilinear product in $\M$, homogeneous with even degree. 
    The pair $(\J,\M)$ is called a $\J$-ternary superalgebra if, for all $a\in \J$ and $m,n,r,s,t\in \M$ (homogeneous), hold the following axioms:
    \begin{enumerate}[label=\textbf{{SJT\arabic*}}]
        \item $a\cdot(m\star n) = \frac{1}{2}\big((a\bullet m)\star n + (-1)^{|a||m|}m\star(a\bullet n)\big)$,\label{SJT1}

        \item $a\bullet(m,n,s) = (a\bullet m,n,s) - (-1)^{|a||m|}(m,a\bullet n,s) + (-1)^{|a|(|m|+|n|)}(m,n,a\bullet s)$,\label{SJT2}
        
        \item $(m,n,s) = (-1)^{|m|(|n|+|s|)+|n||s|}(s,n,m) - (-1)^{|n||s|}(m\star s)\bullet n$,\label{SJT3}

        \item $(m,n,s) = (-1)^{|m||n|}(n,m,s) + (m\star n)\bullet s$,\label{SJT4}

        \item $(m,n,r)\star s + (-1)^{|r|(|m|+|n|)}r\star(m,n,s) = (-1)^{|n|(|r|+|s|)}m\star((r\star s)\bullet n)$,\label{SJT5}

 \item $(m,n,(r,s,t)) = 
 ((m,n,r),s,t) +(-1)^{|r|(|m|+|n|)+|m||n|}(r,(n,m,s),t) + 
                  (-1)^{(|m|+|n|)(|r|+|s|)}(r,s,(m,n,t))$.\label{SJT6}
    \end{enumerate}
\end{Def}

\begin{rem}
Suppose that all operations $\cdot=\bullet=\star$ are identically zero, then a minimal
algebraic manipulation shows that  axioms \eqref{SJT1}-\eqref{SJT6}, in that case, are equivalent to:
\begin{enumerate}\label{c3 triv-op}
    \item $(m,n,s) =(-1)^{|m||n|}(n,m,s)
=    (-1)^{|n||s|}(m,s,n)$,
     
    \item $(m,n,(r,s,t)) = ((m,n,r),s,t) + (-1)^{|r|(|m|+|n|)+|m||n|}(r,(n,m,s),t)+(-1)^{(|m|+|n|)(|r|+|s|)}(r,s,(m,n,t))$.
\end{enumerate}
Therefore, the ternary product is completely (super) symmetric and one can see it as a map
$(-,-,-):S^3\M\to \M$. If $\M=M_0\oplus M_1$ then:
\[
(-,-,-):S^3\M\cong
S^3M_0\ \oplus \ S^2M_0\ot M_1 \ \oplus\ M_0\ot \Lambda^{2}(M_1)\ \oplus\  \Lambda^3M_1
\ \to \M
\]
\end{rem}

\begin{teo}\label{Teor: (sJ,sM) as sJ-tern}
    If $\g  $ is a super short $\sl_2$ Lie algebra
     with isotypic decomposition \eqref{eq: sl2-corta tipo 1} then $(\J,\M)$ is a $\J$-ternary superalgebra with the triple product $(-,-,-):\M \ot \M \ot \M \to \M$ given by:
    \begin{align}\label{eq: STrProd-Jordan}
        (m,n,s):= -\pg_{m,n}(s) + \tfrac{1}{2}(m\star n)\bullet s
    \end{align}

\end{teo}

\begin{rem}In  Lie theoretical terms, we have
$(m,n,s)=-[[m,fn],s]$
\end{rem}

\begin{proof}
\[
(m,n,s)
= -\pg_{m,n}(s) + \tfrac{1}{2}(m\star n)\bullet s
\]
\[
= -\frac12[[m,fn]-[fm,n],s] + 
\tfrac{1}{2}[m, n]\bullet s
= -\frac12[[m,fn]-[fm,n],s]  
-\tfrac{1}{2}[(f[m, n]), s]
\]
\[
= -\frac12
\Big(
[[m,fn],s]+[[fm,n],s]  
-[[fm, n], s]
-[[m, fn], s]\Big)
= -
[[m,fn],s]
\]
\end{proof}

\begin{proof}[Proof of Theorem \ref{Teor: (sJ,sM) as sJ-tern}]

Property \eqref{SJT1} is the same as equation 
\eqref{eq: a.(m*n)}.
For  \eqref{SJT2}:
 \[
 a\bullet(m,n,s)
  - (-1)^{|a|(|m|+|n|)}(m,n,a\bullet s)
  =\]
  \[
=  -  a\bullet [[m,fn],s]
   +(-1)^{|a|(|m|+|n|)}[[m,fn],a\bullet s]
  \]
\[
  =
  [fa,[[m,fn],s]]
  -(-1)^{|a|(|m|+|n|)}[[m,fn],[fa,s]]
  \]
 Using Jacobi in the first summand
\[
  =
  [  [fa,[m,fn]] ,  s]
+(-1)^{|a|(|m|+|n|)}  [[m,fn],  [fa, s]]
  -(-1)^{|a|(|m|+|n|)}[[m,fn],[fa,s]]
  \]
  \[
  =
  [  [fa,[m,fn]] ,  s]
  \]
  and now using Jacobi for $[fa,[m,fn]]$
  \[
  =
  [  [[fa,m],fn] ,  s]
+(-1)^{|a||m|}  [  [m,[fa,fn]] ,  s]
  \]
  
\noindent On the other side,
\[
  (a\bullet m,n,s) - (-1)^{|a||m|}(m,a\bullet n,s)
= \]
\[
=- [[a\bullet m,fn],s] +(-1)^{|a||m|}[[m,f(a\bullet n)],s]
\]
\[
=  [[[fa,m],fn],s] + (-1)^{|a||m|}[[m,f[a,fn]],s]
\]  
as desired.    
    For \eqref{SJT3}
      \[
  (-1)^{|m|(|n|+|s|)+|n||s|}(s,n,m) -(m,n,s)
 =\]
 \[
  =-(-1)^{|m|(|n|+|s|)+|n||s|}[[s,fn],m]
+[[m,fn],s]
  \]
(using Jacobi identity)
  \[
=(-1)^{|n||s|}      [[m,s],fn]
\]
\[
=  (-1)^{|n||s|}[m\star s, fn]
=   (-1)^{|n||s|}(m\star s)\bullet n
\]

\noindent For \eqref{SJT4}:
    
   \[
       (-1)^{|m||n|}(n,m,s) + (m\star n)\bullet s 
   =\]
   \[
   =-(-1)^{ |m||n|}[[n,fm],s]+
   [m,n]\bullet s
  \]
  \[
   =-(-1)^{ |m||n|}[[n,fm],s]-
   [f[m,n], s]
   \]
  \[
   =-(-1)^{ |m||n|}[[n,fm],s]
   -
   [[fm,n], s]
  -[[m,fn], s]
   \]
  \[
  =
  -[[m,fn], s]
   \]
  \[
  = (m,n,s)
   \]
   
\noindent
    For \eqref{SJT5}, $(m,n,r)\star s 
+ (-1)^{|r|(|m|+|n|)}r\star(m,n,s) \overset{?}{=}
(-1)^{|n|(|r|+|s|)}m\star((r\star s)\bullet n)$,
looking at the LHS:
\[
   (m,n,r)\star s + (-1)^{|r|(|m|+|n|)}r\star(m,n,s) =
\]
\[
=
-[[[m,fn],r],s] - (-1)^{|r|(|m|+|n|)}[r,[[m,fn],s] ] 
 \]
On the other side,
\[
(-1)^{|n|(|r|+|s|)}m\star((r\star s)\bullet n)
=(-1)^{|n|(|r|+|s|)}[m, ([r,s]\bullet  n)]
\]
\[
=(-1)^{|n|(|r|+|s|)}[m, [[r,s],fn]]
\]
  and using Jacobi identity, this is
\[
=
(-1)^{|n|(|r|+|s|)}[[m, [r,s]] , fn]
+
[[m,fn],[r,s]]
\]
But notice that $m,r,s$ have eigenalue 1 with respect to $h$,
so $[m,[r,s]]$ has eigenvalue 3 for $h$, hence
$[m,[r,s]]=0$ and we get

\[
(-1)^{|n|(|r|+|s|)}m\star((r\star s)\bullet n)=
[[m,fn],[r,s]]
\]
and Jacobi again
\[
= [ [ [m,fn],r] ,s]
+(-1)^{mr+nr}[r, [[m,fn],s]]
\]

\noindent For \eqref{SJT6}, we need to see
\[(m,n,(r,s,t)) \overset{?}{=} 
 ((m,n,r),s,t) +(-1)^{|r|(|m|+|n|)+|m||n|}(r,(n,m,s),t) + 
                  (-1)^{(|m|+|n|)(|r|+|s|)}(r,s,(m,n,t))
                  \]
We follow the argument in \cite{EBCC23}.
Using the definition
\[
(m,n,s)= -\pg_{m,n}(s) + \tfrac{1}{2}(m\star n)\bullet s\]
 we have
 \[
 (m,n,(r,s,t))
 =
 -\pg_{m,n}(r,s,t) + \tfrac{1}{2}(m\star n)\bullet (r,s,t)
 \]
and for the RHS

\[
 ((m,n,r),s,t) +(-1)^{|r|(|m|+|n|)+|m||n|}(r,(n,m,s),t) + 
                  (-1)^{(|m|+|n|)(|r|+|s|)}(r,s,(m,n,t))
 \]
\[
= \big(( -\pg_{m,n}(r) + \tfrac{1}{2}(m\star n)\bullet r),s,t\big) 
+(-1)^{|r|(|m|+|n|)+|m||n|}
\big(r,( -\pg_{n,m}(s) + \tfrac{1}{2}(n\star m)\bullet s),t\big) 
\]
\[+ 
(-1)^{(|m|+|n|)(|r|+|s|)}
\big(r,s,( -\pg_{m,n}(t) + \tfrac{1}{2}(m\star n)\bullet t)\big)
\]
                  
\[
=
 -\big(\pg_{m,n}(r),s,t\big) 
+ \tfrac{1}{2}\big( (m\star n)\bullet r,s,t\big) 
\]
\[
-(-1)^{|r|(|m|+|n|)+|m||n|}
\big(r,\pg_{n,m}(s) ,t\big) 
+(-1)^{|r|(|m|+|n|)+|m||n|}
+\tfrac12\big(r,(n\star m)\bullet s,t\big) 
\]
\[
- 
(-1)^{(|m|+|n|)(|r|+|s|)}
\big(r,s,\pg_{m,n}(t) \big)
+ 
(-1)^{(|m|+|n|)(|r|+|s|)}
\tfrac12\big(r,s,(m\star n)\bullet t\big)
\]
Using \eqref{SJT2} for $a=\frac12m\star n=(-1)^{|m||n|}\frac12 n\star m$ and the fact that  $\pg_{m,n}=(-1)^{|m||n|}\pg_{n,m}$,
we need  to prove

\[
\pg_{m,n}(r,s,t)\overset{?}{=}
\big(\pg_{m,n}(r),s,t\big) 
+(-1)^{|r|(|m|+|n|)}
\big(r,\pg_{m,n}(s) ,t\big) 
+(-1)^{(|m|+|n|)(|r|+|s|)}
\big(r,s,\pg_{m,n}(t) \big)
\]
For notational convenience, let us denote $d:=\pg_{m,n}$, that is, we want to prove
\[
d(r,s,t)\overset{?}{=}
\big(d(r),s,t\big) 
+(-1)^{|r|(|m|+|n|)}
\big(r,d(s) ,t\big) 
+(-1)^{(|m|+|n|)(|r|+|s|)}
\big(r,s,d(t) \big)
\]
Using again the definition
\[
(r,s,t)=-\pg_{r,s}(t)+\frac12(r\star s)\bullet t
\]
the above equation translates into
\[
-d\big(\pg_{r,s}(t)\big)+
\frac12d\big((r\star s)\bullet t\big)\overset{?}{=}
-\pg_{d(r),s}(t)+\frac12(d(r)\star s)\bullet t
\]
\[
-(-1)^{|r|(|m|+|n|)}
\pg_{r,d(s)}(t)
+(-1)^{|r|(|m|+|n|)}
\frac12(r\star d(s))\bullet t
\]
\[
-(-1)^{(|m|+|n|)(|r|+|s|)}
\pg_{r,s}(d(t))+(-1)^{(|m|+|n|)(|r|+|s|)}\frac12(r\star s)\bullet d(t)
\]
But we know 
(see \eqref{eq: d(a dot b)},\eqref{eq: d(a bul m)}, \eqref{eq: d(m star n)})
that $d$ derives $\star$ and $\bullet$, that is
\[
d\big((r\star s)\bullet t\big)=
(d(r)\star s)\bullet t
+(-1)^{|r|(|m|+|n|)}
(r\star d(s))\bullet t
+(-1)^{(|m|+|n|)(|r|+|s|)}
(r\star s)\bullet d(t)
\]
So, we need to prove

\[
d\big(\pg_{r,s}(t)\big)\overset{?}{=}
\pg_{d(r),s}(t)
+(-1)^{|r|(|m|+|n|)}
\pg_{r,d(s)}(t)
+(-1)^{(|m|+|n|)(|r|+|s|)}
\pg_{r,s}(d(t))
\]
and this is precisely \ref{ganchogancho2}.
\end{proof}

\subsection{The Lie algebra $\Innder(\J,\M)$
and TKK construction for $\J$-ternary algebras}

In this section, $\big(\J,\M,\cdot,\bullet,\star,(-,-,-)\big)$ will be an arbitrary 
$\J$-ternary superalgebra, not necessarily a one coming 
from a super $\sl_2$-Lie algebra.
We begin by extending $\langle a,b\rj :\J\to\J$
to a map $\langle a,b\rjm :\J\oplus\M\to\J\oplus \M$
(in terms of $\cdot$ and $\bullet$) and
define a map $\partial_{m,n}
=\pjm_{m,n}:\J\oplus\M\to\J\oplus\M$ 
(in terms of $\star$, $\bullet$ and the triple product);
we also prove their main properties.

\begin{defi}\label{defpartial}
    Let $(\J,\M)$ be a $\J$-ternary superalgebra with
     operations $\cdot: \J\ot \J \to \J$, 
    $\bullet: \J\ot \M \to \M$, $\star: \M\ot \M \to \J$ and triple
     product $(-,-,-):\M^{\ot 3} \to \M$. 
     Consider elements $a,b,c,d\in \J$ and 
     $m,n,r,s\in\M$ (homogeneous of some degree).
{\bf Define}
$
\langle a,b\rjm , \pjm_{m,n}
\in\End(\J\oplus \M)$
 by
    \begin{enumerate}[label={$\mathsection$\roman*}]
        \item $\langle a,b\rjm  \act c :=\langle a,b\rj  \act c := a\cdot(b\cdot c) - (-1)^{|a||b|}b\cdot(a\cdot c)\in \J$,\label{itm: C1-Der(J+M)}

        \item $\langle a,b\rjm  \act m :=
        \frac14\big( a\bullet(b\bullet m) - (-1)^{|a||b|}b\bullet(a\bullet m)\big)\in \M$,\label{itm: C2-Der(J+M)}

        \item $2\pjm_{m,n}\act a := (-1)^{|a|(|m|+|n|)}(a\bullet m)\star n - (-1)^{|a||n|}m\star(a\bullet n)\in \J$,\label{itm: C3-Der(J+M)} or, under the convention
        $m\bullet a:=(-1)^{|a||m|}a\bullet m$,
\[
\pjm_{m,n}\act a =-\frac12\big(
m\star(n\bullet a)- (-1)^{|m|n|}n\star(m\bullet a)\big)
\]

        \item $\pjm_{m,n}\act{r} := \frac{1}{2}(m\star n)\bullet r - (m,n,r)\in \M$.\label{itm: C4-Der(J+M)}
    \end{enumerate}
\end{defi}

\begin{rem}
One can clearly see that, for $a,b\in\J$ and $m,n\in \M$
\[
\langle a,b\rjm=-(-1)^{|a||b|}
\langle b,a\rjm\]
\[
\pjm_{m,n}(a)=(-1)^{|m||n|}
\pjm_{n,m}(a)
\]
But also, if $r\in\M$ we have
\[
\pjm_{m,n}(r)=(-1)^{|m||n|}
\pjm_{n,m}(r)
\]
\end{rem}
\begin{proof}
By \ref{SJT4}
\[
(m,n,r)=(-1)^{|m||n|}(n,m,r)+(m\star n)\bullet r
\]
and using the definition of $\pjm$
\[
(m,n,r)=\tfrac{1}{2}(m\star n)\bullet r 
- \pjm_{m,n}\act{r}
\]
we conclude
\[
\overbrace{
\tfrac{1}{2}(m\star n)\bullet r 
- \pjm_{m,n}\act{r}}^{(m,n,r)}
=(-1)^{|m||n|}
\overbrace{\Big(
\tfrac{1}{2}(n\star m)\bullet r 
- \pjm_{n,m}\act{r}\Big)}^{(n,m,r)}
+(m\star n)\bullet r
\]
\[
=(-1)^{|m||n|}
\tfrac{1}{2}(n\star m)\bullet r 
- (-1)^{|m||n|}\pjm_{n,m}\act{r}
+(m\star n)\bullet r\]
Since $m\star n=-(-1)^{|m||n|}n\star m$, we have
\[
\tfrac{1}{2}(m\star n)\bullet r 
- \pjm_{m,n}\act{r}
=-\tfrac{1}{2}(m\star n)\bullet r 
- (-1)^{|m||n|}\pjm_{n,m}\act{r}
+(m\star n)\bullet r\]
hence,
$ \pjm_{m,n}\act{r}
=(-1)^{|m||n|}\pjm_{n,m}\act{r}$.
\end{proof}

\begin{lem}\label{lema3.3}
The maps $\langle a,b\rjm $ and $\pjm_{m,n}$
    superderive 
    the operations $\cdot,\bullet$ and $\star$. More
    precisely,
 \[
D (c\cdot d)=
D (c)\cdot d
+(-1)^{|D||c|}
 c\cdot D ( d)
 \]
\[
D (c\bullet r)=
D (c)\bullet r
+(-1)^{|D||c|} c\bullet D ( r)
 \]
\[
D (r\star s)=
D (r)\star s+(-1)^{|D||r|}
 r\star D ( s)
 \]
where either $D=\langle a,b\rjm $ and $|D|=|a|+|b|$
or $D=\pjm_{m,n}$ and $|D|=|m|+|n|$.
\end{lem}
\begin{proof}
    We know that $\langle a,b\rjm$ superderives to $\cdot:\J\ot\J \to \J$. On other hand, since $\M$ is a special $\J$-supermodule we have
    \begin{align*}
        4\langle a,b\rjm\act{c}\bullet m &= a\bullet (b\bullet (c\bullet m)) + (-1)^{|a|(|b|+|c|)+|b||c|}c\bullet (b\bullet (a\bullet m))-(-1)^{|a||b|}b\bullet (a\bullet (c\bullet m))\\
        &\quad -(-1)^{(|a|+|b|)|c|}c\bullet (a\bullet (b\bullet m))
    \end{align*}
    \begin{align*}
        (-1)^{(|a|+|b|)|c|}4c\bullet\langle a,b\rjm\act{m} &= (-1)^{(|a|+|b|)|c|}c\bullet(a\bullet(b\bullet m))-(-1)^{(|a|+|b|)|c|+|a||b|}c\bullet(b\bullet(a\bullet m))
    \end{align*}
    Therefore,
    \begin{align*}
        4\langle a,b\rjm\act{c}\bullet m + (-1)^{(|a|+|b|)|c|}4c\bullet\langle a,b\rjm\act{m} &= a\bullet (b\bullet (c\bullet m))-(-1)^{|a||b|}b\bullet(a\bullet(c\bullet m))\\
        &= 4\langle a,b\rjm\act{c\bullet m}
    \end{align*}
    Similarly, from the equation \eqref{itm: C1-Der(J+M)} and \eqref{SJT1} it follows  that
    \begin{align*}
        \langle a,b\rjm\act{m\star n} = \langle a,b\rjm\act{m}\star n + (-1)^{(|a|+|b|)|m|}m\star \langle a,b\rjm\act{n}.
    \end{align*}
    Hence, $\langle a,b\rjm$ derives 
    $(\J\oplus \M,\cdot,\bullet,\star)$. On other hand, 
    from \eqref{itm: C3-Der(J+M)} and \eqref{SJT1} we
     see that:
    \begin{align*}
        &\pjm_{m,n}\act{a}\cdot b + (-1)^{|a|(|m|+|n|)}a\cdot \pjm_{m,n}\act{b}\\
        &= (-1)^{(|a|+|b|)(|m|+|n|)+|a||b|}\tfrac{1}{4}(b\bullet (a\bullet m))\star n - (-1)^{(|a|+|b|)|n|+|a||b|}\tfrac{1}{4}m\star(b\bullet(a\bullet n))\\
        &\quad + (-1)^{(|a|+|b|)(|m|+|n|)}\tfrac{1}{4}(a\bullet(b\bullet m))\star n - (-1)^{(|a|+|b|)|n|}\tfrac{1}{4}m\star (a\bullet(b\bullet n))\\
        &= \pjm_{m,n}\act{a\cdot b}
    \end{align*}
    Applying the equation \eqref{itm: C4-Der(J+M)} to \eqref{SJT5} we have:
    \begin{align*}
        \tfrac{1}{2}((m\star n)\bullet r)\star s - \pjm_{m,n}\act{r}\star s + (-1)^{|r|(|m|+|n|)}\tfrac{1}{2} &r\star((m\star n)\bullet s)-(-1)^{|r|(|m|+|n|)}r\star\pjm_{m,n}\act{s}\\
        &= (-1)^{|n|(|r|+|s|)}m\star((r\star s)\bullet n) \quad (*)
    \end{align*}
    By the super-commutativity in $\cdot:\J\ot\J\to \J$ and the axiom \eqref{SJT1} we note
    \begin{align*}
        ((m\star n)\bullet r)\star s + &(-1)^{|r|(|m|+|n|)}r\star((m\star n)\bullet s) =(-1)^{(|r|+|s|)(|m|+|n|)}(r\star s)\cdot(m\star n)\\
        &= (-1)^{(|r|+|s|)(|m|+|n|)}((r\star s)\bullet m)\star n + (-1)^{(|r|+|s|)|n|}m\star((r\star s)\bullet n)
    \end{align*}
    Therefore, the equation $(*)$ is equivalent to:
    \begin{align*}
        2\pjm_{m,n}\act{r}\star s + (-1)^{|r|(|m|+|n|)}2r\star\pjm_{m,n}\act{s} &= (-1)^{(|r|+|s|)(|m|+|n|)}((r\star s)\bullet m)\star n\\ 
        &\quad + (-1)^{(|r|+|s|)|n|}m\star((r\star s)\bullet n)\\
        &= 2\pjm_{m,n}\act{r\star s}
    \end{align*}
    Finally, from \eqref{SJT4} and \eqref{itm: C4-Der(J+M)} we get:
    \begin{align*}
        2\pjm_{m,n}\act{r} = -(m,n,r)-(-1)^{|m||n|}(n,m,r)
    \end{align*}
    which is super-symmetric in $m, n\in \M$. By this equation together \eqref{SJT2}:
    \begin{align*}
        2\pjm_{m,n}&\act{a\bullet r}-2\pjm_{m,n}\act{a}\bullet r-(-1)^{|a|(|m|+|n|)}2a\bullet \pjm_{m,n}\act{r}\\
        &= -(m,n,a\bullet r) - (-1)^{|m||n|}(n,m,a\bullet r) - (-1)^{|a|(|m|+|n|)}((a\bullet m)\star n)\bullet r\\
        &\quad + (-1)^{|a||n|}(m\star(a\bullet n))\bullet r + (-1)^{|a|(|m|+|n|)}\big(a\bullet(m,n,r) + (-1)^{|m||n|}a\bullet(n,m,r)\big)\\
        &= -(-1)^{|a|(|m|+|n|)}((a\bullet m)\star n)\bullet r + (-1)^{|a||n|}(m\star(a\bullet n))\bullet r + (-1)^{|a|(|m|+|n|)}(a\bullet m,n,r)\\
        &\quad - (-1)^{(|a|+|n|)|m|}(n,a\bullet m,r) - (-1)^{|a||n|}(m,a\bullet n,r) + (-1)^{|a|(|m|+|n|)+|m||n|}(a\bullet n,m,r)\\
        &= 0
    \end{align*}
    where the final equality is given by \eqref{SJT4}. Therefore, $\pjm_{m,n}$ also super-derives $(\J\oplus \M,\cdot,\bullet,\star)$. 
\end{proof}   

Moreover, the derivations $\pjm_{m,n}$ also 
``derives themselves'', in the following sense:
\begin{lem}\label{lemaDpartial}
Let $\big(\J,\M,\cdot,\bullet,\star,(-,-,-)\big)$
be a $\J$-ternary superalgebra, $m,n,r,s,t\in\M$, and denote 
$D:=\pjm_{m,n}|_\M:\M\to\M$. Then
\[
D(\pjm_{r,s}(t))=
\pjm_{D(r),s}(t)
+(-1)^{|D||r|}\pjm_{r,D(s)}(t)
+(-1)^{|D|(|r|+|s|)}\pjm_{r,s}(D(t))
\]
\end{lem}

\begin{proof}
Just reverse the arguments in the proof of Theorem 
\ref{Teor: (sJ,sM) as sJ-tern}, the part concerning
condition \ref{SJT6}.
\end{proof}

Similarly, $\langle a,b\rjm$ ``derives'' the operation
$m\ot n\ot r\mapsto \pjm_{m,n}(r)$. More precisely:
\begin{lem}\label{lemaDangle}
Keep notation as in the previous lemma. Let $a,b\in\J$, $m,n,r\in\M$ 
and consider $D:=\langle a,b\rjm|_\M:\M\to\M$, then
\[
D(\pjm_{m,n}(r))=
\pjm_{D(m),n}(r)
+(-1)^{|D||m|}\pjm_{m,D(n)}(r)
+(-1)^{|D|(|m|+|n|)}\pjm_{m,n}(D(r))
\]

\end{lem}
\begin{proof}
Recall from \ref{SJT2} we have
\[
a\bullet(m,n,s) = (a\bullet m,n,s) - (-1)^{|a||m|}(m,a\bullet n,s) + (-1)^{|a|(|m|+|n|)}(m,n,a\bullet s)
\]
for any $a\in \J$, $m,n,r\in \M$.
Recall also by definition 
\[
4\langle a,b\rjm (m)=a\bullet (b\bullet m)-
(-1)^{|a||b|}b\bullet(a\bullet m)
\]
Then is clear that, for $D=\langle a,b\rjm$, the following equality holds:
\[
D (m,n,r) =
 \big(D ( m),n,r\big) 
 + (-1)^{(|a|+|b|)|m|}\big(m,D ( n),r\big) 
 + (-1)^{(|a|+|b|)(|m|+|n|)}\big(m,n,D ( r)\big)
\]
Now we use the definition of $\pjm$:
\[
\pjm_{m,n}\act{r} = \frac{1}{2}(m\star n)\bullet r - (m,n,r)
\]
or requivalently $(m,n,r)=\frac{1}{2}(m\star n)\bullet r 
- \pjm_{m,n}\act{r}$, so
\begin{align*}
\frac{1}{2}D\big((m\star n)\bullet r\big)
-D\big(\pjm_{m,n}(r)\big) &=
 \frac12(D(m)\star n)\bullet\act{r}
- \pjm_{D ( m),n}(r)\\
&
+ (-1)^{(|a|+|b|)|m|}
\big(\pjm_{m,D(n)}(r)- \frac12(m\star D( n))\bullet\act{r}
\big)
\\
&
+(-1)^{(|a|+|b|)(|m|+|n|)}\big(
\pjm_{m,n}(D(r))- \frac12(m\star n)\bullet\act{D(r)}
\big)
\end{align*}
But from the previous Lemma we know that $D$ of
$\star$ and $\bullet$, so (and using antisymetry of 
$\star$)
\[
-D \big(\pjm_{m,n}(r)\big)=
- \pjm_{D ( m),n}(r)
- (-1)^{(|a|+|b|)|m|}
\pjm_{m,D(n)}(r)
- 
(-1)^{(|a|+|b|)(|m|+|n|)}
\pjm_{m,n}(D(r))
\]
\end{proof}

\subsubsection{Bracket formulas of the inner derivations}

Now we can conclude the following bracket formulas
for this type of derivations:

\begin{coro}\label{corobracket}
    Let $(\J,\M)$ be a $\J$-ternary superalgebra with
    operations $\cdot: \J\ot \J \to \J$, 
    $\bullet: \J\ot \M \to \M$, $\star: \M\ot \M \to \J$ and triple
    product $(-,-,-):\M^{\ot 3} \to \M$. 
    Then, the vector subspace of $\End(\J\oplus\M)$
spanned by the endomorphisms $\langle a,b\rjm$
and $\pjm_{m,n}$ is closed under (super) 
bracket (with respect to composition), that is,  they form a Lie
subalgebra.
 Moreover, if $D=\langle a',b'\rjm$ or $D=\pjm_{m',n'}$ then
 \begin{align*}
    [D, \langle a, b\rjm ] 
    &= \big\langle D(a),b\big\rjm
    +(-1) ^{|a||D|} \big\langle a,
   D(b)\big\rjm\\
    [D, \pjm_{m,n} ] 
    &= \pjm_{D(m),n}
    +(-1) ^{|m||D|} 
    \pjm_{m,D(n)}
    \end{align*}
    \end{coro}

\begin{proof}
Notice that if one define 
$m\bullet a:=(-1)^{|a||m|}a\bullet m$, then the definition of
\[
\langle a,b\rjm \act c = a\cdot(b\cdot c) - (-1)^{|a||b|}b\cdot(a\cdot c)\]
\[
4\langle a,b\rjm \act m = a\bullet(b\bullet m) - (-1)^{|a||b|}b\bullet(a\bullet m)
\]
\[2\pjm_{m,n}\act a = (-1)^{|a|(|m|+|n|)}(a\bullet m)\star n - (-1)^{|a||n|}m\star(a\bullet n)
\]
can be compacly written as
\[
D_{z_1,z_2}(z_3)=z_1*(z_2*z_3)-(-1)^{|z_1||z_2|}z_2*(z_1*z_3)
\]
for $z_1,z_2,z_3$ equal to $a,b,c$, or $a,b,m$, or
$m,n,a$, and $*$ the convenient operation $\cdot,\bullet$ or $\star$. Now if $\d:\J\oplus\M\to\J\oplus \M$ is another endomorphism that super derives the operations $\cdot,\bullet,\star$, then it is a standard computation that
\[
[\d,D_{z_1,z_2}](z_3)=\d\Big(z_1*(z_2*z_3)-(-1)^{|z_1||z_2|}(z_2*(z_1*z_3)\Big)
\]
\[
-
(-1)^{|\d|(|z_1|+|z_2|)}\Big(z_1*(z_2*d(z_3))-(-1)^{|z_1||z_2|}(z_2*(z_1*d(z_3))\Big)
\]
\[
=
\Big(
D_{d(z_1),z_2}
+(-1)^{|\d||z_1|}D_{z_1,d(z_2)}\Big)(z_3)
\]
So, it remains to prove the cases
\begin{align}
  [\langle a, b\rjm, \pjm_{m,n} ] (m')
  &\overset{?}{=} \pjm_{\langle a,b\rjm(m),n}(m')
    +(-1) ^{|m|(|a|+|b|)} 
    \pjm_{m,\langle a,b\rjm(n)}(m')
   \label{b1} \\
[\pjm_{m,n},\pjm_{r,s}](m')
  &  \overset{?}{=}\pjm_{\pjm_{m,n}(r),s}(m')
      +
      (-1) ^{|r|(|m|+|n|)} 
    \pjm_{r,\pjm_{m,n}(r)}(m')
\label{b2}
\end{align}
 For that, we consider
  $D:\M\to\M$ satisfying
 \[
 D\big(\pjm_{r,s}(t)\big)=
\pjm_{D(r),s}(t)
+(-1)^{|D||r|}\pjm_{r,D(s)}(t)
+(-1)^{|D|(|r|+|s|)}\pjm_{r,s}\big(D(t)\big)
\]
If this is the case,  then one easily see that
 \[
 [D,\pjm_{r,s}](m')=
 D\big(\pjm_{r,s}(m')\big)
 -(-1)^{|D|(|r|+|s|)}\pjm_{r,s}\big(D(t)\big)
 =
 \pjm_{D(r),s}(m')
+(-1)^{|D||r|}\pjm_{r,D(s)}(m')
\]
Now \eqref{b1} and
\eqref{b2}  are true
thanks to Lemmas \ref{lemaDangle} and \ref{lemaDpartial}, using 
$D=\langle a,b\rjm$ and $D=\pjm_{m,n}$ respectively.
\end{proof}    
    
\begin{defi}\label{Def: InnDer(J,M)}
The sub-vector space of $\End(\J\oplus\M)$
spanned by $\langle a,b\rjm$ and $\pjm_{m,n}$ 
will be called the Lie algebra of Inner Derivations
of $\big(\J,\M,\cdot,\bullet,\star,(-,-,-)\big)$ and will
be denoted by $\InnDer(\J,\M)$.
\end{defi}

\subsubsection{$\Tkk(\J,\M)$}
The so-called $\Tkk$-construction for Jordan algebras
is given by the following (see \cite{S22} for an
exposition in the super case):

\[
\Tkk(\J):=\sl_2\ot\J\oplus\InnDer(\J)
\]
We generalize the construction super $J$-ternary objects
$\big(\J,\M,\cdot,\bullet,\star,(-,-,-)\big)$.

\begin{Teor}\label{Teor: sJ-tern to til(g)}
    If $(\J,\M)$ is a $\J$-ternary superalgebra with operations $\cdot: \J\ot \J \to \J$, $\bullet: \J\ot \M \to \M$, $\star: \M\ot \M \to \J$ and triple product $(-,-,-):\M^{\ot 3} \to \M$ then the vector superspace
    \[
    \g(\J,\M) = \sl_2\ot \J \; \oplus \; V\ot \M \; \oplus \; \Innder(\J,\M)
    \]
    is a short $\sl_2$-Lie superalgebra with brackets \eqref{eq: Brk-g-1}-\eqref{eq: Brk-g-5},
\end{Teor}

\begin{Proof} We will check the conditions of Theorem
\ref{Teo: STits-1}.
Equation \eqref{eq: <ab,c>} follows from the fact
 that $(\J,\cdot)$ is a super Jordan algebra.
 Equation  \eqref{eq: a.(m*n)} is equivalent to \eqref{SJT1}
  and equation \eqref{eq: partial_{m,n}-2} is equivalent to
   \eqref{SJT3}. Therefore, we only need to prove 
   equation \eqref{eq: angle-partial}, which we will do
    it in two parts. First, note that \eqref{SJT1} and the 
    $\J$-supermodule structure of $\M$ implies:
    \begin{align*}
        4a\cdot ((m\star n)\cdot b) &= (-1)^{|b|(|m|+|n|)}
        (a\bullet(b\bullet m))\star n 
        + (-1)^{|b|(|m|+|n|)+|a|(|b|+|m|)}
        (b\bullet m)\star(a\bullet n)\\
        &\quad + (-1)^{|b||n|}(a\bullet m)\star(b\bullet n)
        +(-1)^{|b||n|+|a||m|)}m\star(a\bullet(b\bullet n))
    \end{align*}
    \begin{align*}
        4(m\star n)\cdot (a\cdot b) &= 
        (-1)^{|b|(|m|+|n|)}(a\bullet(b\bullet m))\star n 
+ (-1)^{|b|(|m|+|n|)+|a||b|}(b\bullet(a\bullet m))\star n\\
&\quad + (-1)^{|a||m|+|b||n|}m\star(a\bullet(b\bullet n))
+(-1)^{|a|(|m|+|b|)+|b||n|}m\star(b\bullet(a\bullet n))
    \end{align*}
    Therefore, by \eqref{itm: C1-Der(J+M)} and
     \eqref{itm: C3-Der(J+M)} we get:
    \begin{align*}
 4\langle a,m\star n\rjm\act{b} 
+ 2\pjm_{a\bullet m,n}\act{b}
-(-1)^{|a||m|}2\pjm_{m,a\bullet n}\act{b} = 0
    \end{align*}
    and \eqref{eq: angle-partial} holds for $\J$. On other hand,
     if $r\in \M$ then from \eqref{itm: C2-Der(J+M)} and
      \eqref{itm: C4-Der(J+M)} we get:
    \begin{align*}
        4\langle a,m\star n\rjm\act{r} &
        + 2\pjm_{a\bullet m,n}\act{r}
        -(-1)^{|a||m|}2\pjm_{m,a\bullet n}\act{r}\\
        &= a\bullet((m\star n)\bullet r)
        -(-1)^{|a|(|m|+|n|)}(m\star n)\bullet(a\bullet r)
        +((a\bullet m)\star n)\bullet r-2(a\bullet m,n,r)\\
        &\quad -(-1)^{|a||m|}(m\star(a\bullet n))\bullet r+(-1)^{|a||m|}2(m,a\bullet n,r)\\
        &=0
    \end{align*}
    where the final equality holds by \eqref{SJT4} and
     \eqref{SJT2} respectively. Therefore, 
     \eqref{eq: angle-partial} also holds for
      $\M$ and $\g(\J,\M)$ is a $\sl_2$-Lie superalgebra.
\end{Proof}

\begin{Rem}
    Let $\g  \in \ShsL$ and $(\J,\M)$ be the $\J$-ternary superalgebra obtained by Theorem \ref{teogsl2ternary}, then $\g $ and $\g(\J,\M)$ given by Theorem \ref{Teor: sJ-tern to til(g)} are not necessarily isomorphic.
    \begin{equation*}
    \xymatrix{
	\g  \ar@{~>}[rr]^{\text{Theor. \ref{teogsl2ternary}\qquad}} & & (\J,\M) 
    \ar@{~>}[rr]^{\hskip -0.5cm\text{Theor. \ref{Teor: sJ-tern to til(g)}}\qquad} & &	\g(\J,\M)=TKK(\J,\M)\\	
    }
    \end{equation*}
\end{Rem}

\section{The functors $\Tag$ and $\T$ extended to $\ShsL$ and $\JterS$ respectively}\label{Sec: Tag(J,M)-T(g)}

\subsection{$\Tag(\J)$: The TAG construction for super Jordan algebras
}\label{Subsec: Tag(J)}

The TKK construction for Jordan algebras
 and
for J-ternary algebras
\[
\Tkk(\J):=\sl_2\ot\J\oplus\Innder(\J),\hskip 1cm
\Tkk(\J,\M):=\sl_2\ot\J\oplus V\ot \M\oplus \Innder(\J,\M)
\]
are ``natural'' in a non technical sense, since they are not functorial. For instance, if $\J$ is a commutative and {\em associative} algebra that we see it as a Jordan algebra, then
for every $a,b\in\J$, the maps $\langle a,b\rj $ are
 the zero maps on $\J$. But if
$\J$ is a Jordan subalgebra of some bigger non-associative
Jordan algebra $\wt \J$, then $\langle a,b\rangle_{\wt\J} $
 may be non-zero on $\wt \J$. This problem of non functoriality for
$\Tkk(\J)$ is solved by the 
Tits-Allison-Gao construction $\Tag(\J)$.
In this section we will revise TAG construction for (super) Jordan algebras $\J$
and we will generalize it for (super) ternary algebras
$(\J,\M,(-,-,-))$.

The main ingredients is the \eqref{eq: <ab,c>} from Theorem
\ref{Teo: STits-1}
and an organizing argument given by  in \cite{S22}
that we will reproduce here in Lemma \ref{Lem: Bs-3} next.
 Recall:

\[
\langle a, b\rj(c)=a\cdot (b\cdot c)-(-1)^{|a||b|}b\cdot(a\cdot c)
\]
clearly satisfies $\langle a,b\rj=-(-1) ^{|a||b|}
\langle b,a\rj$, and \eqref{eq: <ab,c>}  says
\[
             \langle a\cdot b,c\rg  + (-1)^{|c|(|b|+|a|)}\langle c\cdot a,b\rg  + (-1)^{|a|(|c|+|b|)}\langle b\cdot c,a\rg  = 0
      \]
The TAG construction consists on considering
abstract symbols $\{a\ot b\}$
that will act on $\J$ as if they were $\langle a,b\rj$.
Because of the previous identities satisfied by $\langle a,b\rj$,
it is natural to consider the super vector space
$\J\ot\J$
and   $R^s(\J)$ be the subspace ``of relations''
   of $\J\ot\J$ spanned by
\begin{itemize}
    \item[] $I_{a,b} = a\ot b + (-1)^{|a||b|}b\ot a$,

    \item[] $I_{a,b,c} = (a\cdot b)\ot c + (-1)^{|c|(|a|+|b|)}(c\cdot a)\ot b + (-1)^{|a|(|b|+|c|)}(b\cdot c)\ot a$,
\end{itemize}
for all homogeneous elements $a,b,c\in \J$.
Define $\B^s(\J)$ as the vector superspace 
\[
\B^s(\J) = (\J \ot \J) / R^s(\J).
\]
We denote by $\{a\ot b\}$ the canonical image of $a\ot b$ in $\B^s(\J)$. Note that the first generator is defined from the fact that $\langle a,b\rj$ is super-antisymmetric and the second generator is defined from the equation \eqref{eq: <ab,c>}. Following \cite{S22} we get the following properties:

\begin{Lem}\label{Lem: Bs-1}
Under the natural action of $\Der(\J)$ on $\J\ot\J$:
    \[
    D\act{a\ot b} = D\act a\ot b + (-1)^{|D||a|}c\ot D\act b,
    \]
the map
    $\varphi: \B^s(\J) \to \Innder(\J)$ given by $\{ a\ot b\} \mapsto \langle a,b\rj$ is  well-defined, surjective and 
a $\Der(\J)$-graded module map. In particular, it is a graded 
$\Innder(\J)$-module map.
\end{Lem}
\begin{proof}
    By definition,  $\Innder(\J)$ is generated by the derivations $\langle a,b\rj$, so:
    \begin{align*}
        \varphi: \J\ot \J \to \Innder(\J), \{ a\ot b\}\mapsto \langle a,b\rj 
    \end{align*}
    is surjective, and clearly graded.
 Additionall, using $\langle a,b\rj=-(-1) ^{|a||b|}\langle b,a\rj$ and
 \eqref{eq: <a,b>}, it is clear that
    \begin{align*}
        \varphi(I_{a,b}) = 0,\quad \varphi(I_{a,b,c}) = 0,
    \end{align*}
    which implies that $R^{s}(\J)$ is contained in $\Ker(\varphi)$ and the map $\varphi: \B^{s}(\J) \to \Innder(\J)$ is well-defined.
Also, if $D\in\Der(\J)$ is a homogeneous derivation of degree $|D|$, one can explicitely check that
\[
D(I_{a,b})
=  D(a\ot b) + (-1)^{|a||b|}D(b\ot a)
\]
\[
=  D(a)\ot b
+(-1)^{|D||a|}a\ot D(b)
 + (-1)^{|a||b|}D(b)\ot a
 + (-1)^{|a||b|+|D||b|}b\ot D(a)
\]
reordering terms and using $|D|+|a|=|D(a)|$ and
$|D|+|b|=|D(b|)$
\[
=  D(a)\ot b
 + (-1)^{|a||b|+|D||b|}b\ot D(a)
+(-1)^{|D||a|}
\big(
a\ot D(b)
 + (-1)^{|D(b)||a|}D(b)\ot a
\big)
\]
\[
= I_{ D(a),b}+
(-1)^{|D||a|}I_{a,D(b)}
 \]
Similarly (using that $D$ derives $\cdot$)
\[
D(I_{a,b,c})
=
I_{D(a),b,c}
+(-1)^{|a||D|}I_{a,D(b),c}
+(-1)^{(|a|+|b|)|D|}I_{a,b,D(c)}
\]
so $R^s(\J)$ is a $\Der(\J)$-submodule, hence
$\B^s(\J)$ is also $\Der(\J)$-submodule and the induced
map $\B^s(\J)\to\InnDer(\J)$ is $\Der(\J)$-linear.
\end{proof}

Now we use the organizing argument  for Lie algebras:

\begin{Lem}[Lemma 3.11 in \cite{Smi11}bla]
\label{Lem: Bs-3}
    Let $M$ be a graded module over a Lie superalgebra $\g$ and let $\lambda: M \to \g$ be a $\g$-morphism. Then:
    \[
       A(M) = \langle \lambda(m)\cdot n + (-1)^{|m||n|} \lambda(n) \cdot m: \; \textrm{$m,n$ are homogeneous in $M$} \rangle 
    \]
    is a submodule of $M$ such that $\Img(\lambda) \cdot \Ker(\lambda) \subset A(M) \subset \Ker(\lambda)$, where $\lambda(m)\cdot n$ denotes the $\g$-action of $\lambda(m)$ on $n$. The quotient module $Q=M/A(M)$ defines a Lie superalgebra via the bracket $[p,q] = \lambda(p)\cdot q$ for all $p, q \in Q$ and the map $\mu: Q \to \g$ induced by $\lambda$ is a central extension of $\g$.
\end{Lem}

For $M=\B^s(\J)$ and $\g=\InnDer$ we get
that $A(M)$ is the subspace of $\B^s(\J)$ generated by
elements of the form
\[
\langle a,b\rj\big(\{c\ot d\}\big)
+(-1)^{(|a|+|b|)(|c|+|d|)}\langle c,d\rj\big(\{a\ot b\}\big)
\]
so, $Q=B^s/A(B ^s)$ is a Lie algebra with the desired properties.

\begin{lem}\label{lema4.4}
If $M=\B^s(\J)$ and $\g=\InnDer$
as above, then $A(M)=0$.
\end{lem}

\begin{proof}It is given in \cite{S22}. For convenience of the reader
we include the non super version. The super version can be recovered just adding Koszul signs.
In we consider $J$ a Jordan algebra concenrtated in degree zero, then
$R^s(J)$ is spanned by elements of the form
\[
I_{a,b}=a\ot b+b\ot a, \hskip 1cm
I_{a,b,c}=a\cdot b\ot c+
b\cdot c\ot a+
c\cdot a\ot b
\]
and we want to see that
\[
\langle a,b\rangle (c\ot d)+
\langle c,d\rangle(a\ot b) \overset{?}{\in} R^s(J)
\]
We compute
\[
(*):=\langle a,b\rangle(c\ot d)
+
\langle c,d\rangle(a\ot b)
\]
\[
=\langle a,b\rangle(c)\ot d
+c\ot \langle a,b\rangle(c)
+\langle c,d\rangle(a)\ot b
+a\ot\langle c,d\rangle(b)
\]
\[
=
a\cdot (b\cdot c)\ot d
-b\cdot (a\cdot c)\ot d
+c\ot a\cdot(b\cdot d)
-c\ot b\cdot(a\cdot d)
\]
\[
+c\cdot (d\cdot a)\ot b
-d\cdot (c\cdot a)\ot b
+a\ot c\cdot (d\cdot b)
-a\ot d\cdot (c\cdot b)
\]
Using $I_{\alpha,\beta}$ one can
order terms by $\alpha\ot\beta \equiv -\beta\ot\alpha\Mod R^s$ and get
\[
(*)\equiv
a\cdot (b\cdot c)\ot d
-b\cdot (a\cdot c)\ot d
- a\cdot(b\cdot d)\ot c
+ b\cdot(a\cdot d)\ot c
\]
\[
+c\cdot (d\cdot a)\ot b
 -d\cdot (c\cdot a)\ot b
- c\cdot (d\cdot b)\ot a
- d\cdot (c\cdot b)\ot a
\]
Using $I_{a,b\cdot c,d}$ one has
\[
a\cdot (b\cdot c)\ot d=I_{a,b\cdot c,d}-(b\cdot c)\cdot d\ot a
-(a\cdot d)\ot (b\cdot c)
\]
and similarly
\[-b\cdot (a\cdot c)\ot d
=-I_{b,a\cdot c,d} +(a\cdot c)\cdot d\ot b+(b\cdot d)\ot (a\cdot c)\]
\[- a\cdot(b\cdot d)\ot c
=-I_{a,b\cdot d,c}
+(b\cdot d)\cdot c\ot a+(a\cdot c)\ot (b\cdot d)\]
\[
 b\cdot(a\cdot d)\ot c
=I_{b,a\cdot d,c}
-(a\cdot d)\cdot c\ot b-(b\cdot c)\ot (a\cdot d)
\]
so,
\[
(*)=
I_{a,b\cdot c,d}
-(a\cdot d)\ot (b\cdot c)
-I_{b,a\cdot c,d} 
+(b\cdot d)\ot (a\cdot c)
-I_{a,b\cdot d,c}
+(a\cdot c)\ot (b\cdot d)
I_{b,a\cdot d,c}
-(b\cdot c)\ot (a\cdot d)
\]
\[
=I_{a,b\cdot c,d}
-I_{b,a\cdot c,d} 
-I_{a,b\cdot d,c}
I_{b,a\cdot d,c}
-I_{a\cdot d,b\cdot c}
+I_{a\cdot c,b\cdot d}
\equiv 0
\]

\end{proof}

\begin{Cor}\label{Cor: Bs-1}
$\B^s(\J)$ is a Lie superalgebra with
well-defined bracket given by
\[
\big[\{a\ot b\},\{c\ot d\}\big]
=\langle a,b\rj\big(\{c\ot d\}\big)
=(-1) ^{(|a|+|b|)(|c|+|d|)}\langle c,d\rj\big(\{a\ot b\}\big)
\]
and $\phi: \B^s(\J) \to \Innder(\J)$ is a central extension of Lie superalgebras.
\end{Cor}

\hfill

Following \cite{S22}, we define $\Tag(\J)$ as the vector superspace
\begin{align}
    \Tag(\J) = \Sl_2\ot \J \oplus \B^s(\J)
\end{align}
and we will consider on $\Tag(\J)$ the following brackets:
\begin{enumerate}
    \item $[\{ a\ot b\}, \{ c\ot d\}] 
    = \langle a,b\rj(\{ c\ot d\})$

    \item $[x\ot a,y\ot b] = [x,y]\ot a\cdot b + \frac{1}{2}\killing(x,y)\{ a\ot b\}$,

    \item $[\{ a\ot b\}, x\ot c] = x\ot \langle a,b\rj\act{c}$.
\end{enumerate}
where $x,y\in\sl_2$, $a,b,c \in \J$. 
Note that $\Tag(\J)$ is an object in $\VShsL$ with:
\[
\Tag(\J)^{ad} = \sl_2 \otimes \J, \quad \Tag(\J)^{\sl_2} = \B^{s}(\J)
\]
In other words, we declare 
$\langle a,b\rangle_{{}_{\Tag(\J)}}:=\{a\ot b\}$.

\begin{Teor}
    $\Tag(\J)$ is a Lie superalgebra in the category $\VShsL$.
\end{Teor}

A well-known fact is:
\begin{Teor}
The Tits-Allison-Gao functor 
 $\Tag: \Jsa \to \VShsL$ is left adjoint to the Tits functor 
 $\T: \VShsL \to \Jsa$:
 \[
 \Hom_{\VShsL}(\Tag(\J), \g) =\Hom_{\VShsL}(\sl_2\ot \J\oplus \B^s(\J), \g)
\cong
 \Hom_{\Jsa}(\J, \T(\g))
\]
\end{Teor}

\subsection{$\Tag(\J,\M)$:
The TAG construction for super $\J$-ternary 
 algebras}

Let $\JterS$ be the category whose objects are the $\J$-ternary superalgebras $(\J,\M)$ given in Subsec. \ref{Subsec: J-ternary super}. 
That is, $\J$ is a Jordan superalgebra with product 
$\cdot:\J\ot \J\to \J$ and $\M$ is a special $\J$-supermodule
 with action $\bullet:\J\ot \M\to \M$ equipped with a 
 super-symmetric operation $\star:\M\ot \M\to \J$ and a triple
  product $(-,-,-):\M\ot\M\ot\M \to \M$ with certain
   compatibilities.
One shoud write $\big(\J,\M,\cdot,\bullet,\star,(-,-,-)\big)$
but by abuse of notation, if the context is clear,
 we will write simply $(\J,\M)$.

If $(\J,\M)$ and $(\wt\J,\wt\M)$ are two $\wt\J$-ternary superalgebra, a morphism $\eta:(\J,\M)\to (\wt\J,\wt\M)$ in $\JterS$ is a pair $(\eta_1,\eta_2)$ with $\eta_1:\J\to \wt\J$ and $\eta_2:\M\to \wt\M$ satisfying the following conditions:
\begin{enumerate}[label=\roman*]
    \item $\eta_1$ is a morphism of Jordan superalgebras
that is $\eta_1(a\cdot b)=\eta_1(a)\cdot \eta_1(b)$.
    \item $\eta_1$ and $\eta_2$ satisfy:
    \begin{align*}
        \eta_2(a\bullet m) &= \eta_1(a)\bullet\eta_2(m)\\
        \eta_1(m\star n) &= \eta_2(m)\star \eta_2(n).
    \end{align*}
    \item $\eta_2$ is compatible with the triple product of $\M$ and $\wt\M$, that is:
    \begin{align*}
        \eta_2(m,n,r) = (\eta_2(m),\eta_2(n),\eta_2(r))
    \end{align*}
\end{enumerate}

\subsection{The Lie superalgebra $\B ^s(\J,\M)$}\label{SubSec: Tag(J,M)}

We will generalize the construction $\Tag(\J)$ mentioned in
 \ref{Subsec: Tag(J)} when $(\J,\M)$ is a $\J$-ternary
  superalgebra. Recall  (see Definition \ref{defpartial}) 
  $\Innder(\J,\M)$ is generated by  
 derivations of the form $\langle a,b\rjm$ 
   and $\pjm_{m,n}$given by:
\begin{align}
\langle a,b\rjm \act c &= a\cdot(b\cdot c) - (-1)^{|a||b|}b\cdot(a\cdot c)\label{itm: D1-Der(J+M)}
\\
4\langle a,b\rjm \act m &= a\bullet(b\bullet m) - (-1)^{|a||b|}b\bullet(a\bullet m)\label{itm: D2-Der(J+M)}
\\
2\pjm_{m,n}\act a &= (-1)^{|a|(|m|+|n|)}(a\bullet m)\star n - (-1)^{|a||n|}m\star(a\bullet n)\label{itm: D3-Der(J+M)}
\\
\pjm_{m,n}\act{r} &= \tfrac{1}{2}(m\star n)\bullet r - (m,n,r)\label{itm: D4-Der(J+M)}
\end{align}

Recall also from \eqref{ganchogancho} that
\[
[\pjm_{m,n},\pjm_{r,s}]
= \pjm_{\pjm_{m,n}(r),s} 
+ (-1)^{(|m|+|n|)|r|}\pjm_{r,\pjm_{m,n}(s)}
\]
and because
$
[\pjm_{m,n},\pjm_{r,s}]
=-(-1)^{(|m|+|n|)(|r|+|s|}[\pjm_{r,s},\pjm_{m,n}]
$, 
then necessarily
\begin{align}\label{itm: D5-Der(J+M)}
 \pjm_{\pjm_{m,n}(r),s} 
+ (-1)^{(|m|+|n|)|r|}\pjm_{r,\pjm_{m,n}(s)}
+
(-1)^{(|m|+|n|)(|r|+|s|)}
\big( \pjm_{\pjm_{r,s}(m),n} 
+ (-1)^{(|r|+|s|)|m|}\pjm_{m,\pjm_{r,s}(n)}
\big)=0
\end{align}

We need functorial replacement of the derivations 
$\langle a,b\rjm$ and $\pjm_{m,n}$, so we will consider abstract symbols $\{a\ot b\}$ as in TKK construction and we also need
symbols $\{m\ot n\}$ for $\pjm_{m,n}$. So, motivated by the symmetry properties and
\eqref{itm: D1-Der(J+M)}, \eqref{itm: D2-Der(J+M)},
\eqref{itm: D3-Der(J+M)},
\eqref{itm: D4-Der(J+M)},
\eqref{itm: D5-Der(J+M)},
we consider
$\J\ot\J\oplus \M\ot\M$ and $R^s(\J,\M)$ the
``relations''
 subspace of $\J\ot \J \;\oplus\; \M\ot\M\;$ 
 generated by elements of the form:
\begin{itemize}
    \item[$I_{a,b}$:] $ a\ot b + (-1)^{|a||b|}b\ot a$,

    \item[$I_{m,n}$:] $m\ot n - (-1)^{|m||n|}n\ot m$,
    
    \item[$I_{a,b,c}$:] $ (-1)^{|a||c|}(a\cdot b)\ot c + (-1)^{|a||b|}(b\cdot c)\ot a + (-1)^{|b||c|}(c\cdot a)\ot b$,

    \item[$I_{a,m,n}$:] $ (a\bullet m)\ot n - (-1)^{|a||m|}m\ot (a\bullet n) - (-1)^{|a|(|m|+|n|)}2(m\star n)\ot a$,
   \item[$I_{m,n,r,s}$:] $\pjm_{m,n}(r)\ot s 
+ (-1)^{(|m|+|n|)|r|}r\ot \pjm_{m,n}(s)
+
(-1)^{(|m|+|n|)(|r|+|s|)}
\pjm_{r,s}(m)\ot n 
+ (-1)^{(|r|+|s|)|n|}m\ot \pjm_{r,s}(n)
$.
    
\end{itemize}
where $a,b,c\in \J$, $m,n,r,s\in \M$ are homogeneous.
We define $\B^{s}(\J,\M)$ as the quotient
\[
\B^s(\J,\M) = (\J\ot \J\ \oplus \ \M\ot \M) / R^s(\J,\M)
\]
We will  denote by $\{a\ot b\}$ and $\{m\ot n\}$ the classes of $a\ot b$ and $m\ot n$ respectively in $\B^{s}(\J,\M)$.
We will prove that $\B^s(\J,\M)$ is naturally a Lie superalgebra,
acting by derivations on $\J$ and $\M$.

\begin{Lem}\label{Lem: Bs-1 JM}
     $\varphi: \B^s(\J,\M) \to \Innder(\J,\M)$ given by $\{a\ot b\} \mapsto \langle a,b\rjm $, $\{m\ot n\} \mapsto \pjm_{m,n}$ is surjective and $(\J\ot \J\ \oplus \ \M\ot \M)$ is a graded $\Innder(\J,\M)$-module by the natural action of $\Innder(\J,\M)$ on $\J$ and $\M$.
\end{Lem}

\begin{Proof}
    We know that $\langle -,-\rjm : \J\ot \J\to \Innder(\J,\M)$ and $\pjm: \M\ot \M\to \Innder(\J,\M)$ are bilinear maps and generate $\Innder(\J,\M)$. So:
    \begin{align*}
        \varphi: \J\ot \J\ \oplus \ \M\ot \M \to \Innder(\J,\M),\;\; \{a\ot b\}\mapsto \langle a,b\rjm,\; \{m\ot n\}\mapsto \pjm_{m,n}
    \end{align*}
    is surjective. It remains to see the generators of $R^{s}(\J,\M)$ are in $\Ker(\varphi)$. For $I_{a,b}$ and $I_{a,b,c}$ see Lemma \ref{Lem: Bs-1} and for $I_{m,n}$ and $I_{a,m,n}$ note that:
    \begin{itemize}
        \item $\varphi\act{I_{m,n}} = \pjm_{m,n} - (-1)^{|m||n|}\pjm_{n,m} = 0$ by the super-symmetry of $\pjm_{m,n}$.

        \item $\varphi\act{I_{a,m,n}} = \pjm_{a\bullet m,n} - (-1)^{|a||m|}\pjm_{m,a\bullet n} - (-1)^{|a|(|m|+|n|)}2\langle m\star n,a\rjm  = 0$ by \eqref{eq: angle-partial}.
    \end{itemize}
    Therefore $R^{s}(\J,\M)\subset \Ker(\varphi)$ and $\varphi:\B^{s}(\J,\M)\to \Innder(\J,\M)$ is surjective. Moreover, $\J\ot\J \;\oplus\; \M\ot\M$ is a $\Innder(\J,\M)$-module since $\Innder(\J,\M)$ acts on $\J$ and $\M$ by super-derivations.
\end{Proof}

\begin{Lem}\label{Lem: Bs-2 JM}
    $R^s(\J,\M)$ is a graded submodule of $\J\ot \J\ \oplus \ \M\ot \M$ and $\B^s(\J,\M)$ is a $\Innder(\J,\M)$-module.
\end{Lem}

\begin{Proof}
    We need to show that $R^{s}(\J,\M)$ is invariant under the action of $\Innder(\J,\M)$. Similarly to the proof of Lemma
   \ref{Lem: Bs-1}, if
   $D:\J\oplus\M\to \J\oplus \M$ is such that $D$
derives $\cdot,\bullet$ and $\star$
and the ternary operation $r\ot s\ot t\mapsto 
\pjm_{r,s}(t)$ (see Lemmas \ref{lema3.3},
 \ref{lemaDpartial} and \ref{lemaDangle}),
 one can explicitely
see that
    \begin{align*}
        D\act{I_{a,b}} &= I_{D(a),b} + (-1)^{|a||D|}I_{a,D(b)}\\
 D\act{I_{a,b,c}}&=  I_{D(a),b,c} 
 + (-1)^{|D||a|}I_{a,D(b),c} + (-1)^{|D|(|a|+|b|)} I_{a,b,D(c)}
\\
        D\act{I_{m,n}} &= I_{D(m),n} + (-1)^{|m||D|}I_{m,D(n)}\\
        D\act{I_{a,m,n}}&= I_{D(a),m,n} + (-1)^{|D||a|} I_{a,D(m),n} + (-1)^{|D|(|m|+|a|)} I_{a,m,D(n)}\\
        D\act{I_{m,n,r,s}}&= I_{D(m),n,r,s} + (-1)^{|D||m|} I_{m,D(n),r,s} + (-1)^{|D|(|m|+|n|)} I_{m,n,D(r),s}
        + (-1)^{|D|(|m|+|n|+|r|)} I_{m,n,r,D(s)}
    \end{align*}
In particular, $R^{s}(\J,\M)$ is invariant under the 
action of $\langle a,b\rjm$ and  $\pjm_{m,n}$.
 Hence, $\B^{s}(\J,\M)$ is an $\Innder(\J,\M)$-module.
\end{Proof}

\begin{prop}With notation as in Lemma
\ref{Lem: Bs-3},
$A(B^s(\J,\M))=0$.
\end{prop}

\begin{Proof}    
Recall the notation in Lemma \ref{Lem: Bs-3},
for $\lambda:M\to \g$, the subspace $A(M)$
is given by
\[
A(M)=\langle \lambda(\mu)\nu
+(-1)^{|\mu||\nu|}\lambda(\nu)\mu: \mu,\nu\in M\rangle
\]
Here $M=\B^s(\J,\M)$, $\g=\InnDer(\J,\M)$ and $\lambda$
is determined by
\[
\{a\ot b\}\mapsto\langle a,b\rjm
\]
\[
\{m\ot n\}\mapsto\langle m,n\rjm
\]
So, for $\mu,\nu$ elements of the form 
$\{a\ot b\}$, 
$\{c\ot d\}$,
 $\{m\ot n\}$
 $\{r\ot s\}$
 (where $a,b,c,d\in\J$ and $m,n,r,s\in\M$)
  we need to prove
\[
\lambda(\{ a\ot b\})\{c\ot d\}
+(-1)^{(|a|+|b|)(|c|+|d|)}
\lambda(\{ c\ot d\})\{a\ot b\}\overset{?}{=}0
\]
\[
\lambda(\{ a\ot b\})\{m\ot n\}
+(-1)^{(|a|+|b|)(|m|+|n|)}
\lambda(\{ m\ot n\})\{a\ot b\}\overset{?}{=}0
\]
\[
\lambda(\{ m\ot n\})\{r\ot s\}
+(-1)^{(|r|+|s|)(|m|+|n|)}
\lambda(\{ r\ot s\})\{m\ot n\}\overset{?}{=}0
\]
That is, we need to prove
\[
 \langle a,b\rjm(c\ot d) 
 + (-1)^{(|a|+|b|)(|c|+|d|)}
 \langle c,d\rjm (a\ot b) \in R^{s}(\J,\M)
 \]
\[
 \langle a,b\rjm(m\ot n) 
 + (-1)^{(|a|+|b|)(|m|+|n|)}
 \pjm_{m,n}(a\ot b) \in R^{s}(\J,\M)
 \]
\[
 \pjm_{m,n}(r\ot s) + (-1)^{(|m|+|n|)(|r|+|s|)}\pjm_{r,s}(m\ot n) \in R^{s}(\J,\M)
 \]
 The first assertion is Lemma \ref{lema4.4}, while the last one is just $I_{m,n,r,s}$. 
 So, let us consider $a,b\in \J$, $m,n\in \M$, 
 we will show that:
    \begin{align*}
        (*) = \langle a,b\rjm(m\ot n) + (-1)^{(|a|+|b|)(|m|+|n|)}\pjm_{m,n}(a\ot b) \equiv 0 \mod R^{s}(\J,\M).
    \end{align*}
By the definition of the action of $\InnDer(\J,\M)$ we have:
    \begin{align*}
        (*) &= \langle a,b\rjm(m)\ot n + (-1)^{(|a|+|b|)|m|}m\ot \langle a,b\rjm(n)\\
        &\qquad\quad + (-1)^{(|a|+|b|)(|m|+|n|)}\pjm_{m,n}(a)\ot b + (-1)^{|b|(|m|+|n|)}a\ot \pjm_{m,n}(b)
    \end{align*}
    Hence, using the definition of $\langle-,-\rjm$ and $\pjm$ (see \eqref{itm: D2-Der(J+M)} and \eqref{itm: D3-Der(J+M)}) we obtain:
    \begin{align*}
        (*) &= \tfrac{1}{4}a\bullet(b\bullet m)\ot n - (-1)^{|a||b|}\tfrac{1}{4}b\bullet(a\bullet m)\ot n + (-1)^{(|a|+|b|)|m|}\tfrac{1}{4}m\ot a\bullet(b\bullet n)\\
        &\quad - (-1)^{(|a|+|b|)|m|+|a||b|}m\ot b\bullet(a\bullet n) + (-1)^{|b|(|m|+|n|)}\tfrac{1}{2}(a\bullet m)\star n\ot b\\
        &\quad - (-1)^{|b|(|m|+|n|)+|a||m|}\tfrac{1}{2}m\star(a\bullet n)\ot b + \tfrac{1}{2}a\ot (b\bullet m)\star n\\
        &\quad - (-1)^{|b||m|}\tfrac{1}{2}a\ot m\star(b\bullet n)
    \end{align*}
    From $I_{a,b\bullet m,n}\equiv 0 \mod R^s(\J,\M)$ we see that:
    \begin{align*}
        a\bullet(b\bullet m)\ot n = (-1)^{|a|(|b|+|m|)}b\bullet m\ot a\bullet n + (-1)^{|a|(|b|+|m|+n)}\tfrac{2}{\mu}(b\bullet m)\star n\ot a
    \end{align*}
    and substituting into $(*)$ we obtain:
    \begin{align*}
        (*) &= (-1)^{|a|(|b|+|m|)}\tfrac{1}{4}b\bullet m\ot a\bullet n - (-1)^{|a||b|}\tfrac{1}{4}b\bullet(a\bullet m)\ot n \\
        &\quad +(-1)^{(|a|+|b|)|m|}\tfrac{1}{4} m\ot a\bullet(b\bullet n) - (-1)^{(|a|+|b|)|m|+|a||b|}\tfrac{1}{4}m\ot b\bullet(a\bullet n) \\
        &\quad + (-1)^{|b|(|m|+|n|)}\tfrac{1}{2}(a\bullet m)\star n\ot b - (-1)^{|b|(|m|+|n|)+|a||m|}\tfrac{1}{2}m\star (a\bullet n)\ot b \\
        &\quad - (-1)^{|b||m|}\tfrac{1}{2}a\ot m\star (b\bullet n)
    \end{align*}
    From $I_{b,a\bullet m,n}\equiv 0 \mod R^s(\J,\M)$ we see that:
    \begin{align*}
        (-1)^{|a||b|}\tfrac{1}{4}b\bullet(a\bullet m)\ot n &= (-1)^{|b||m|}\tfrac{1}{4}a\bullet m\ot b\bullet n\\
        &\quad + (-1)^{|b|(|m|+|n|)}\tfrac{1}{2}(a\bullet m)\star n\ot b
    \end{align*}
So,
    \begin{align*}
        (*) &= (-1)^{|a|(|b|+|m|)}\tfrac{1}{4}b\bullet m\ot a\bullet n - (-1)^{|b||m|}\tfrac{1}{4}a\bullet m\ot b\bullet n \\
        &\quad + (-1)^{(|a|+|b|)|m|}\tfrac{1}{4}m\ot a\bullet(b\bullet n) - (-1)^{(|a|+|b|)|m|+|a||b|}\tfrac{1}{4}m\ot b\bullet(a\bullet n)\\
        &\quad - (-1)^{|b|(|m|+|n|)+|a||m|}\tfrac{1}{2}m\star(a\bullet n)\ot b - (-1)^{|b||m|}\tfrac{1}{2}a\ot m\star(b\bullet n)
    \end{align*}
    From $I_{a,m,b\bullet n}\equiv 0 \mod R^s(\J,\M)$ we see that:
    \begin{align*}
        (-1)^{(|a|+|b|)|m|}\tfrac{1}{4}m\ot a\bullet(b\bullet n) &= (-1)^{|b||m|}\tfrac{1}{4}a\bullet m\ot b\bullet n \\
        &\quad - (-1)^{|a|(|b|+|m|+|n|)+|b||m|}\tfrac{1}{2}m\star(b\bullet n)\ot a
    \end{align*}
    and then,
    \begin{align*}
        (*) &= (-1)^{|a|(|b|+|m|)}\tfrac{1}{4}b\bullet m\ot a\bullet n - (-1)^{|a|(|b|+|m|+|n|)+|b||m|}\tfrac{1}{2}m\star(b\bullet n)\ot a\\
        &\quad - (-1)^{(|a|+|b|)|m|+|a||b|}\tfrac{1}{4}m\ot b\bullet(a\bullet n) - (-1)^{|b|(|m|+|n|)+|a||m|}\tfrac{1}{2}m\star(a\bullet n)\ot b\\
        &\quad - (-1)^{|b||m|}\tfrac{1}{2}a\ot m\star(b\bullet n)
    \end{align*}
    Finally, from $I_{b,m,a\bullet n}\equiv 0 \mod R^s(\J,\M)$ we see that:
    \begin{align*}
        (-1)^{(|a|+|b|)|m|+|a||b|}\tfrac{1}{4}m\ot b\bullet(a\bullet n) &= (-1)^{|a|(|b|+|m|)}\tfrac{1}{4}b\bullet m\ot a\bullet n\\
        &\quad - (-1)^{(|a|+|b|)|m|+|b||n|}\tfrac{1}{2}m\star(a\bullet n)\ot b
    \end{align*}
    and substituting into $(*)$ we obtain:
    \begin{align*}
        (*) &= -(-1)^{|a|(|b|+|m|+|n|)+|b||m|}\tfrac{1}{2}m\star(b\bullet n)\ot a + (-1)^{(|a|+|b|)|m|+|b||n|}\tfrac{1}{2}m\star(a\bullet n)\ot b\\
        &\quad - (-1)^{(|a|+|b|)|m|+|b||n|}\tfrac{1}{2}m\star(a\bullet n)\ot b + (-1)^{|a|(|b|+|m|+|n|)+|b||m|}\tfrac{1}{2}m\star(b\bullet n)\ot a\\
        &\equiv 0 \mod R^s(\J,\M)
    \end{align*}
\end{Proof}

As immediate consequence we have the following:

\begin{Teor}\label{Teor: Bs-1 JM}
    $\B^s(\J,\M)$ is a Lie superalgebra with 
    brackets given by
    \begin{align*}
   [\{a\ot b\}, \{c\ot d\}] &:= \langle a,b\rjm\act{\{c\ot d\}}\\
   [\{a\ot b\}, \{m\ot n\}] &:= \langle a,b\rjm(\{m\ot n\})\\
   [\{m\ot n\}, \{r\ot s\}] &:= \pjm_{m,n}\act{\{r\ot s\}}
\end{align*}
and $ \phi: \B^s(\J,\M) \to \Innder(\J,\M)$, determined
 by $ \{a\ot b\}\mapsto\langle a,b\rjm$ 
 and  $\{m\ot n\}\mapsto\pjm_{m,n}$
        is a central extension of Lie superalgebras.
\end{Teor}

\subsection{Definition of $\Tag(\J,\M)$}

\begin{defi}
 We define $\Tag(\J,\M)$ as the vector superspace
\begin{align}\label{eq: Desc-hat(g)2}
    \Tag(\J,\M) = \sl_2\ot \J \; \oplus V\ot\M \; \oplus \; \B^s(\J,\M)
\end{align}
where $\lambda,\mu\in\C$, $|V|=|\sl_2|=0$ and we will define on $\Tag(\J,\M)$ the following brackets:

\begin{enumerate}[label=\roman*]
    \item $[\{ a\ot b\}, \{c\ot d\}] = \langle a,b\rjm(\{c\ot d\})$,

    \item $[\{ a\ot b\}, \{m\ot n\}] = \langle a,b\rjm(\{m\ot n\})$,

    \item $[\{m\ot n\}, \{r\ot s\}] = \pjm_{m,n}(\{r\ot s\})$,

    \item $[\{a\ot b\}, x\ot c] = x\ot \langle a,b\rjm\act{c}$,

    \item $[\{m\ot n\}, x\ot c] = x\ot \pjm_{m,n}\act{c}$,

    \item $[\{a\ot b\}, v\ot m] = v\ot\langle a,b\rjm\act{m}$,

    \item $[\{m\ot n\}, v\ot r] = v\ot\pjm_{m,n}\act{r}$,

    \item $[x\ot a,y\ot b] = [x,y]\ot a\cdot b + \tfrac{1}{2}\killing(x,y)\{ a\ot b\}$.

    \item $[x\ot a,v\ot m] = x(v)\ot a\bullet m$,

    \item $[v\ot m,w\ot n] = v\odot w\ot m\star n + \det(v,w)\{m\ot n\}$
\end{enumerate}
for \(x, y \in \sl_2\), \(v, w \in V\), \(a, b, c \in \J\) and \(m, n, r \in \M\). Note that \(\Tag(\J,\M)\) is an object of \(\ShsL\), where
\[
\Tag(\J,\M)^{ad} = \sl_2 \otimes \J, \quad \Tag(\J,\M)^{\sl_2} = \B^s(\J,\M)
\]
\end{defi}

\begin{rem}
Formulas i-vii can be compactly written as

\begin{enumerate}
    \item[i-iii] $[D, \{z_1\ot z_2\}] =
     \{D(z_1)\ot z_2\}+(-1) ^{ |D||z_1|}
    \{z_1\ot D(z_2)\}
     $, \\
     where 
     $D=\{a\ot\b\}$ acting by 
     $\langle a,b\rjm $ or $D=\{m\ot n\}$ acting 
     by $\pjm_{m,n}$,
     if $z_1,z_2=a,b\in\J$ or $z_1,z_2=r,s\in\M$
     respectively, and
    \item[iv-v] $[D, x\ot a'] = x\ot D\act{a'}$,
    \item[vi-vii] $[D, v\ot m'] = v\ot D\act{m'}$,
where $D\in \B^s(\J,\M)$.
\end{enumerate}

\end{rem}

\begin{Teor}
    $\Tag(\J,\M)$ is a $\sl_2$-Lie superalgebra in the category $\ShsL$.
\end{Teor}

\begin{Proof} 
 It is immediate with the same proof that $\Tkk(\J,\M)$
 is a super Lie algebra (Theorem \ref{Teor: sJ-tern to til(g)})
 and the fact that 
 \begin{align*}
 \B^{s}(\J,\M)&\to \Innder(\J,\M)
 \\
 \{a\ot b\}&\mapsto\langle a,b\rjm
 \\
 \{m\ot n\}&\mapsto\pjm_{m,n}
 \end{align*}
 is a Lie algebra map.
\end{Proof}

The advantage of $\Tag(\J,\M)$ with respect
 to $\Tkk(\J,\M)$ is
it obvious functorial nature.

\begin{Cor}
    $\Tag:\JterS \to \ShsL$ is a functor from the category of $\J$-ternary superalgebras $\JterS$ to the category $\ShsL$.
\end{Cor}

\section{The Tits functor for ternary Jordan super
 algebras\label{Subsec: funtor hat(T)}}

In \cite{S22} is described the super case of the Tits functor $\T=\J(-)$,
which assigns to each very short  super Lie algebra $\g $, 
a Jordan superalgebra $\T(\g )=(\J(\g ),\cdot)$, 
describing its product through the Lie bracket:
\[
    \J (\g):= \{a\in \g  :\: ha=2a\}
\]
\[
a\cdot b:=\tfrac12[a, fb]
\]
where as usual $h$ and $f$ stands for the usual generators of $\sl_2$.
In this section we consider the generalization for
short (not necessarily very short) super Lie algebras, and
 we prove that this functor is right adjoint to  {\em right adjoint}
to  $\Tag(\J,\M)$.
Recall the definition of $\J(\g )$, $\M(\g )$
and the operations $\cdot,\bullet,\star, (-,-,-)$ associated to 
a short $\sl_2$-super Lie algebra $\g$.

\begin{defi}\label{defternary}
Given a short $\sl_2$-Lie superalgebra $\g $, we define the following subspaces:
\begin{align}\label{eq: Desc JM in g}
    \J(\g ) = \{a\in \g  :\: ha=2a\}
    ,\hskip 1cm 
    \M(\g ) = \{m\in \g :\: hm=m\}
\end{align}
where the operations $\cdot:\J\ot \J\to \J$, $\bullet:\J\ot \M\to \M$, $\star:\M\ot \M\to \J$, $(-,-,-):\M^{ \ot 3}\to\M$
 are given by:
\begin{equation}\label{eq: operations on J+M}
\begin{aligned}
    2a\cdot b &:= [a,fb]=-[fa,b],\\
    a\bullet m&:= [a,fm]=-[fa,m],\\
    m\star n  &:= [m,n],\\
    (m,n,r)  &:= -[[m,fn],r]
\end{aligned}
\end{equation}
\end{defi}

\begin{lem}\label{rem9}
From the ternary product and $\cdot,\bullet,\star$ as above
one can recover $\langle-,-\rjm$ and $\pjm$ acting
on $\J(\g)\oplus\M(\g)$
by
\begin{align}\label{eq: partial_{m,n} brk-1}
    \langle a,b\rjm = \tfrac14\Big[[fa,fb],-\Big],
    \hskip 1cm
    \pjm_{m,n} = \tfrac12\Big[[m,fn]-[fm,n],-\Big]
\end{align}
\end{lem}
\begin{proof}
Using the definition of $\langle a,b\rjm$ acting on $\J$,
$\langle a,b\rjm \act c 
= a\cdot(b\cdot c) - (-1)^{|a||b|}b\cdot(a\cdot c)$
\[
=\tfrac14 \big([fa,[fb,c]] -(-1)^{|a||b|}[fb,[fa,c]]\big)
=[\tfrac14 [fa,fb],c]
\]

Using the definition of $\langle a,b\rjm$ acting on $\M$,
$\langle a,b\rjm \act m 
= \tfrac14\big(
a\bullet(b\bullet m) - (-1)^{|a||b|}b\bullet(a\bullet m)$,
\[
= \tfrac14\big(
[fa,fb,m]] - (-1)^{|a||b|}[fb,[fa,m]]
\big)
=[\tfrac14 [fa,fb],m]
\]

By the definition of $\pjm_{m,n}$ acting on $\J$:
\[
\pjm_{m,n}\act a 
=\tfrac12\big(
 (-1)^{|a|(|m|+|n|)}(a\bullet m)\star n - (-1)^{|a||n|}m\star(a\bullet n)
 \big)
 \]
\[
=\tfrac12\big(
 (-1)^{|a|(|m|+|n|)}[[a,fm],n] 
 - (-1)^{|a||n|}[m,[a,f n]]
 \big)
 \]
\[
=\tfrac12\big(
 (-1)^{|n||m|)}[n,[fm,a]] 
+[m,[f n,a]]
 \big)
 \]
  and using $[\M,\J]=0$,
\[
=\tfrac12\big(
 (-1)^{|n||m|)}[n,[fm,a]] 
-[fm,[n,a]] 
+[m,[f n,a]]
-(-1) ^{|m||n|}[fn,[m,a]]
 \big)
 \]
\[
=\tfrac12\big[
[m,f n]
-[fm,n],
a\big]
 \]
And finally, 
with  the definition of $\pjm_{m,n}$ acting on $\M$ we have
$\pjm_{m,n}\act{r} 
= \tfrac{1}{2}(m\star n)\bullet r - (m,n,r)$

\[
= -\tfrac{1}{2}[f[m,n],r] + [[m,fn],r]
\]
\[
= -\tfrac{1}{2}[[fm,n],r]
-\tfrac{1}{2}[[m,fn],r]
 + [[m,fn],r]
\]
\[
= -\tfrac{1}{2}[[fm,n],r]
 +\tfrac12 [[m,fn],r]
=
\tfrac12\big[ [m,fn]-[fm,n],
r\big]
\]
\end{proof}

\begin{Teor}\label{teogsl2ternary}
If $\g $ is a short $\sl_2$-super Lie algebra,
 then     $(\J,\M)$ together  with the above operations
 $\cdot,\bullet,\star$ and $(-,-,-)$ is a $\J$-ternary
  superalgebra.
\end{Teor}
\begin{Proof}
 We already know that if $\g $ is a short $\sl_2$-super Lie algebra then $\g \cong \sl_2\ot\J\oplus V\ot\M\oplus
 \D$, where $\D=\g  ^{\sl_2}$, and $(\J,\M)$
 is a ternary Jordan algebra. So, this theorem is a consequence of the ``dictionary'' between the operations on $\J$
 and $\M$ and
 the $\sl_2$-Lie structure of $\M$ 
 given by Proposition \ref{propdic1}
 and Theorem
 \ref{Teor: (sJ,sM) as sJ-tern}. 
 \end{Proof}

\begin{Def}
For $\g \in \ShsL$,  we define extend the functor
$\T$ into a functor $\hat\T$ by
\[
\hat{\T}(\g )=
\big(\J(\g ),
\M(\g ),\cdot,\bullet,\star,(-,-,-)\big)
\]
as in Definition \ref{defternary}.
\end{Def}

 We prove the following key Lemma to be used in the
 adjoint property to be proved later:
 
\begin{Lem}\label{Lem: phi-Tg 1}
    Let  $\g$ be a short $\sl_2$-super Lie algebra. For
    $a,b\in \J:=\J(\g)=\{\alpha\in\g:h\alpha=2\alpha\}$
   and 
    $m,n\in \M:=\M(\g)=\{\alpha\in\g:h\alpha=\alpha\}$
   {\em define}
\[
\hskip -1cm \langle-,-\rg:\J\ot\J\to\g^{\sl_2}, 
\hskip 5cm
\pg:\M\ot\M\to\g^{\sl_2}
\]
\[
a\ot b\mapsto\langle a,b\rg:=\tfrac14[fa,fb]\in\g^{\sl_2}
\hskip 1cm
m\ot n\mapsto \pg_{m,n}:=\tfrac12\big([m,fn]-[fm,n]\big)\in\g^{\sl_2}
\]
Then, there is a unique morphism of Lie superalgebras $\phi$
    \[
   \Tag(\hat\T(\g ))=\sl_2\ot\J(\g)\ \oplus \ V\!\ot\M(\g)\ \oplus \ \B^s(\J(\g),\M(\g))\overset{\phi}{\longrightarrow} \g 
    \]
    completely determined by
      \begin{align*}
e\ot a &\mapsto a, \\
e_{1}\ot m &\mapsto m\\
\{a\ot b\} &\mapsto \tfrac{1}{4}[fa,fb]=\langle a,b\rg,\\
 \{m\ot n\} &\mapsto \tfrac{1}{2}\big([m,fn]-[fm,n]\big)=\pg_{m,n}.
	\end{align*}
  
\end{Lem}

\begin{rem}The following relation hold
\[
\langle a,b\rjm=\big[\langle a,b\rg,-\big]|_{\J(\g)\oplus\M(\g)}
,\hskip 1cm
\pjm_{m,n}=\big[\pg_{m,n},-\big]|_{\J(\g)\oplus\M(\g)}
\]
\end{rem}

\begin{proof}[Proof of 
Lemma \ref{Lem: phi-Tg 1}]
The Lie bracket in $\Tag(\g)=\sl_2\ot\J(\g)\oplus
V\ot\M(\g)\oplus \B^s(\J(\g),\M(\g))$ is given by
\begin{enumerate}[label=\roman*]
    \item $[\{D\},x\ot a] = x\ot \{D\}\act{a}$,
    \item $[\{D\},v\ot m] = v\ot \{D\}\act{m}$, 
    \item $[x\ot a,y\ot b] = [x,y]\ot a\cdot b + \tfrac{1}{2}\killing(x,y)\{ a,b\}$, \
    \item $[x\ot a,v\ot m] =[x,v]\ot a\bullet m$, 
    \item $[v\ot m,w\ot n] = v\odot w\ot m\star n 
    + \det(v,w)\{m,n\}$, 
    
where $a,b\in\J(\g)$, $m,n\in\M(\g)$, $x,y\in\sl_2$, $v,w\in V$
and $\{D\}\in \B^s(\J(\g),\M(\g))$ is acting by derivations on $\g$
via  $\{a\ot b\}\mapsto\langle a,b\rjm$ and $\{m\ot n\}\mapsto\pg_{m,n}$.

\item $[\{D\},\{z_1\ot z_2\}]=\{\{D\}(z_1)\ot z_2\}+
(-1)^{|D||z_1|}\{z_1\ot \{D\}(z_2)\}$, where $z_1\ot z_2$ 
belongs to $\J\ot\J$ or to $\M\ot\M$.
\end{enumerate}
On the other hand, $\g$ is also short:
\[
\g\cong\sl_2\ot\J(\g)\oplus V\ot\M(\g)\oplus\g^{\sl_2}
\]
so, we know that:
	\begin{align*}
		\g  
		= 
		(\J \oplus \M) \oplus (f\J \oplus f\M) \oplus f^2\J \oplus \g ^{\sl_2} 
		\cong 
		\sl_2\ot\J \oplus \ V\ot\M \ \oplus \ \g ^{\sl_2}
	\end{align*}
Hence, the $\sl_2$-linearity gives that necesarily $\phi$
has to be defined as
  \begin{align*}
		e\ot a \mapsto a, &\quad h\ot a \mapsto -fa,\quad f\ot a \mapsto -\tfrac{1}{2}f^2a,\\
		e_{1}\ot m \mapsto m, &\quad e_{2}\ot m \mapsto fm,\\
		\{a\ot b\} \mapsto \tfrac{1}{4}[fa,fb]=\langle a,b\rg, &\quad \{m\ot n\} \mapsto \tfrac{1}{2}\big([m,fn]-[fm,n]\big)=\pg_{m,n}.
	\end{align*}
It is straighforward to check that $\phi$ is a 
morphism of $\sl_2$-modules.
We know that the map $B^{s}(\J(\g),\M(\g))$
given by $\{a\ot \b\}\mapsto\langle a,b\rjm$ and $\{m\ot n\}\mapsto\pjm_{m,n}$ is a Lie algebra map.
 The rest of the computation that $\phi$ is a Lie algebra map follows from the formulas given in
Proposition \ref{propdic1} and Remark \ref{rem9},
  computing the operations in terms of the bracket and the $\sl_2$ action.
\end{proof}

\begin{rem}\label{Lem: phi-Tg 3}
The morphism $\phi_{\g }:\Tag(\hat\T(\g ))\to \g $ is
defined, by construction, as a 
morphism of $\sl_2$-modules, and being ``the identity'' on $\J$ and $\M$.
\end{rem}

\subsection{Adjunction property of $\T$ and $\Tag$}
\begin{Teor}\label{Teo: adj T-Tag}
	$\Tag:\JterS \to \ShsL$ is the left adjoint of $\hat{\T}:\ShsL \to \JterS$.
\end{Teor}
\begin{Proof}
	We want to prove that:
	\begin{align*}
		\Hom_{\JterS}\big((\J,\M),\hat{\T}(\g )\big) \cong \Hom_{\ShsL}\big(\Tag(\J,\M),\g \big)
	\end{align*}
	for any $(\J,\M)\in \JterS$, $\g \in \ShsL$. Note that $\hat{\T}(\Tag(\J,\M))=(e\ot \J,e_1\ot\M)\cong(\J,\M)$, so there exists a natural morphism
	\begin{align*}
		\mu: \Hom_{\ShsL}\big(\Tag(\J,\M),\g \big) \to \Hom_{\JterS}\big((\J,\M),\hat{\T}(\g )\big)
	\end{align*}
If we identify $e\ot\J\cong \J$ and $e_1\ot \M\cong \M$ then this morphism is just the restriction $\eta \to \eta|_{\J,\M}$. Since $\Tag(\J,\M)$ is generated by $\sl_2\ot\J$ and $V\ot \M$, it follows that $\mu$ is injective. 
	
	On the other hand, let $\phi:(\J,\M) \to (\wt\J,\wt\M)$ be a morphism in $\JterS$, with $(\wt\J,\wt\M)=\Tag(\g )$. By the functoriality of $\Tag$, there exist a morphism of Lie superalgebras
	\begin{align*}
		\Tag(\phi): \Tag(\J,\M) \to \Tag(\wt\J,\wt\M)
	\end{align*}
which also is $\sl_2$-linear. By Lemma 
\ref{Lem: phi-Tg 1}, there exist a canonical morphism of
 Lie superalgebras 
 $\phi_{\g }:\Tag(\hat{\T}(\g )) \to \g $ 
which is $\sl_2$-linear. Hence, $\phi_{\g }\circ \Tag(\phi)$ is
 a morphism in $\ShsL$ that extends $\phi$, so that 
 $\phi = \mu\big(\phi_{\g }\circ \Tag(\phi)\big)$ and 
 $\mu$ is injective.
\end{Proof}

\subsection{Some Examples}

\begin{ex}
Let $\J=0$ and $\M=\C m_0$ with $|m_0|=0$. 
All operations
(including the triple product and
$\pjm_{m_0,m_0}$) are necessarily trivial. 
Hence, $\InnDer(0,\C m_0)=0$ and the TKK construcion 
gives the abelian  Lie algebra $V\ot m_0$.
On the other hand,  
$\B^s(0,\C m_0)=\C\{m_0\ot m_0\}$, and the TAG construction
gives
\begin{align*}
\Tag(0,\C m_0)= V\ot m_0\oplus \C \{m_0\ot m_0\}
&\cong V\oplus \C z =\h_3\\
v\ot m_0&\leftrightarrow v\\
\{m_0\ot m_0\}&\leftrightarrow z
\end{align*}
with bracket
\[
[v\ot m_0,w\ot m_0]=\det(v,w)\{m_0\ot m_0\}
\leftrightarrow [v,w]=\det(v,w)z
\]
And we recover the 3-dimensional Heisenberg
Lie algebtra  $\h_3$.
\end{ex}

\begin{ex}
	Let $\J =J_0= \C^{1|0}=\M=M_0$ with the operations
 $1\cdot 1=1$, $1\bullet 1=1$ and $1\star 1=0$. 
 The axiom \eqref{SJT6} implies that the only triple product
  on $\J$ is $(1,1,1)=0$ and, the associative structure of the
   operations show that $\B^{s}(\J,\J)=0$. Hence, 
$\Tag(\J,\J) = (\sl_2 \oplus V)\ot\J$ is the short $\sl_2$-Lie
 superalgebra with non trivial brackets defined by:
	\begin{align*}
		[x\ot 1,y\ot 1] = [x,y]\ot 1,\quad [x\ot 1,v\ot 1] = x(v)\ot 1
	\end{align*}
	Hence, 
$\Tag\big(
\J,\J,\cdot,\bullet,\star=0,(-,-,-)=0
\big)=\sl_2\ltimes V_{(1)}=\sl_2\ltimes V$.
\end{ex}

\begin{ex}
Let $\J = \C^{1|0}=\C 1$ and
 $\M = \C^{0|1}=\C\wt 1$ with the
 operations $1\cdot 1=1$, $1\bullet \wt1=\wt1$ and 
 $\wt1\star \wt1=1$. The axiom \eqref{SJT3} implies that
  the only non-zero triple product on $\M$ is 
  $(\wt1,\wt1,\wt1)=\tfrac{1}{2}\wt1$.
 The super symmetry relations shows that
   $\B^{s}(\J,\M)=0$: 
\[
I_{1,1}= 1\ot 1 + (-1)^{|1||1|}1\ot 1
\]
\[
I_{\wt 1,\wt 1}= \wt 1\ot \wt 1 -(-1)^{|\wt 1||\wt 1|}\wt 1\ot \wt 1
\]
\[
 \To \{1\ot 1\}= 0=\{\wt 1\ot \wt 1\}
 \]
Hence,  $\B^s(\J,\M)=0$ 
and  $\Tag(\J,\M) = \sl_2\ot\J \; \oplus\; V\ot\M$ is 
   isomorphic to the
  the orthosymplectic Lie superalgebra $\osp(1,2)$ (see \cite{Me2017,Kac78}), with  brackets given by
\[
\begin{array}{rcl}
\Tag(\J,\M) = \sl_2\ot\J \; \oplus\; V\ot\M&\cong &\sl_2\oplus V=\osp(1,2)\\
\sl_2\ot \C^{1|0}\ni x\ot 1 &\leftrightarrow &x\\
V\ot \C^{0|1}\ni v\ot \wt 1 &\leftrightarrow &v\\
{}[x\ot 1,y\ot 1] = [x,y]\ot 1 &\leftrightarrow&{} [x,y]=[x,y]_{\sl_2}\\
 {}[x\ot 1,v\ot \wt1] = x(v)\ot \wt1 &\leftrightarrow &{}[x,v]=x(v)
 \\
{}[v\ot \wt1,w\ot \wt1] = v\odot w\ot 1
&\leftrightarrow& [v,w]=v\odot w
\end{array}
\]
\end{ex}

\begin{ex}\label{ex-sl3}
    Let $E_{ij}$ be the elementary matrices in $M_{3}(\C)$ and $\g = \sl_3$ the semisimple Lie algebra with an ordered basis:
    \[
        \{E_{12}, E_{13}, E_{23}, H_{12} = E_{11}-E_{22}, H_{23} = E_{22}-E_{33}, E_{21}, E_{31}, E_{32}\}
    \]
    By choosing $e=E_{12}$, $f=E_{21}$ and $h=E_{12}$, we have that $\sl_3$ is a $\sl_2$-Lie algebra via adjoint representation and then:
    \begin{align}
	\sl_3 = V_{2} \;\oplus\; V_{1}^1 \oplus V_{1}^2 \;\oplus\; V_{0}
    \end{align}
    where each irreducible $\sl_2$-module, in term of weight basis, is defined by:
    \begin{itemize}
	\item $V_{0} = \langle E_{11}+E_{22}-2E_{33}\rangle$,
		
	\item $V_{2} = \langle E_{12},-H_{12},-2E_{21}\rangle$,
		
	\item $V_{1}^1 = \langle E_{13},E_{23}\rangle$, $ V_{1}^2 = \langle E_{32},-E_{31}\rangle$.
    \end{itemize}
    Define $J:=\langle a\rangle$ and $M:=\langle m_1,m_2\rangle$ for $a=E_{12}$, $m_1=E_{13}$ and $m_2=E_{32}$. A direct computation gives:
    \begin{equation}
	\begin{aligned}
		[E_{12},H_{12}] = -2E_{12},\;\;
		[E_{12},E_{23}] = E_{13},\;\;
		[E_{12},-E_{31}] = E_{32},\;\;
		[E_{13},E_{32}] = E_{12},
	\end{aligned}
    \end{equation}
    and then, the operations \eqref{eq: operations on J+M} results:
    \begin{equation}\label{eq: oper on JM from sl3}
	\begin{aligned}
		a\cdot a = a,\;\;
		a\bullet m_1 = m_1,\;\;
		a\bullet m_2 = m_2,\;\;
		m_1\star m_2 = a.
	\end{aligned}
    \end{equation}
    The triple product $(-,-,-)$ defined in \eqref{eq: operations on J+M} implies that:
    \begin{equation}\label{eq: tr_prod on M from sl3}
	\begin{aligned}
		(m_1,m_2,m_1) &= 2m_1,&\quad (m_1,m_2,m_2) &= -m_2\\
		(m_2,m_1,m_1) &= m_1,&\quad (m_2,m_1,m_2) &= -2m_2
	\end{aligned}
    \end{equation}
    From eq. \eqref{eq: partial_{m,n} brk-1} we have that the only non-zero derivation is (in basis $\{a,m_1,m_2\}$):
    \begin{equation}
	\partial_{12} = 
	\left(\begin{smallmatrix}
		0 & 0 & 0\\
		0 &-3/2 & 0\\
		0 & 0 & 3/2
	\end{smallmatrix}\right),
    \end{equation}
    where $\partial_{12}=\partial_{m_1,m_2}$. It is straightforward to see that $R^s(J,M)$ is generated by $a\ot a$, $m_1\ot m_1$ and $m_2\ot m_2$, which implies that $B^s(J,M)$ is generated by $\{m_1\ot m_2\}$ and $B^{s}(J,M)$ is isomorphic to $\C$. 
    
    Therefore, $\hat{\T}(\sl_3)$ is the $J$-ternary algebra $(J,M)$ with operations \eqref{eq: oper on JM from sl3} and triple product \eqref{eq: tr_prod on M from sl3}, and $\Tag(\hat{\T}(\sl_3))$ is the Lie algebra
    \[
        \Tag(\hat{\T}(\sl_3)) = \sl_2 \; \oplus \; V\ot M \; \oplus \; \C \{m_1\ot m_2\}.
    \]
    isomorphic to $\sl_3$.
\end{ex}

\begin{ex}Let $\h_3$ be the 3-dimensional Heisenberg Lie algebra. 
This algebra is short, $\h_3=V\oplus \C z$, but 
also the complex $\Lambda \h_3\ot\Lambda\h_3^*$
 is a (differential graded)
short $\sl_2$-super Lie algebra:
\[
\Lambda \h_3^*\ot\Lambda\h_3=\Big(\bigoplus_{p,q=0}^3\Lambda^p\h_3^*\ot \Lambda^q\h_3,\;d_{CE}\Big)
\]
Moreover,
$H^{\bullet,\bullet}(\h_3):=H^\bullet(\h_3,\Lambda \h_3)=$ 
 the cohomology of this d.g. Lie algebra is a short
$\sl_2$-super Lie algebra, with
$\J=\Lambda(\C z\oplus\C z^*)$
 a Jordan algebra of dimension 4, 
$\dim\M=16$, and $\dim \D=20$.

\begin{align}\label{eq: short desc of H(h3)}
    H^{\bullet,\bullet}(\h_3)\cong V_{(2)}\ot H^\bullet(\J)\; \oplus\; V_{(1)}\ot H^\bullet(\M)\; \oplus\; V_{(0)}\ot H^\bullet(\D)
\end{align}
\[\h_3=V\oplus\C z\]
\[
\h_3^*=V^*\oplus\C z^*\]
\[
\Lambda \h_3=\underbrace{\Lambda ^2V\wedge z}_{\Lambda ^3\h_3} 
\oplus
\underbrace{
\big( \Lambda^2 V
\oplus
 V\wedge z \big)}_{\Lambda ^2\h_3}
\oplus 
\underbrace{
\big(V\oplus \C z\big)}_{\h_3}
\oplus \C\]
and similarly
\[
\Lambda \h^*_3=\Lambda ^2V^*\wedge z^*\oplus 
\big( \Lambda^2 V^*
\oplus V^*\wedge z^* \big)\oplus \big(V^*\oplus \C z^*\big)\oplus \C\]
Notice that $\Lambda^2V$ and its dual are trivial
 $\sl_2$-modules. As representations of $\sl_2$
  one can rearrange

\[
\Lambda \h_3=
\underbrace{\big(\Lambda ^2V\wedge z
\oplus
 \Lambda^2 V
\oplus
\C z
\oplus \C\big)}_{\hbox{trivial $\sl_2$-module}}
\oplus
\big( 
 V\wedge z 
\oplus 
V 
\big)\]
\[
\Lambda \h_3 ^*=
\big(\Lambda ^2V^*\wedge z^*
\oplus
 \Lambda^2 V^*
\oplus
\C z^*
\oplus \C\big)
\oplus
\big( 
 V^*\wedge z^* 
\oplus 
V^* 
\big)
\]
and finally
\[
\begin{array}{rclc}
\Lambda\h_3^*\ot\Lambda \h_3&=&
\big(\Lambda ^2V^*\wedge z^*
\oplus
 \Lambda^2 V^*
\oplus
\C z^*
\oplus \C\big)
\ot
\big(\Lambda ^2V\wedge z
\oplus
 \Lambda^2 V
\oplus
\C z
\oplus \C\big)
&trivial\\
&&
\oplus
\big(\Lambda ^2V^*\wedge z^*
\oplus  \Lambda^2 V^* \oplus \C z^* \oplus \C\big)
\ot \big(V\wedge z\oplus V)
&V_{(1)} \\
&&\oplus
\big(\Lambda ^2V\wedge z
\oplus
 \Lambda^2 V
\oplus
\C z
\oplus \C\big)
\ot \big(V^*\wedge z^*\oplus V^*)
&V_{(1)} \\
&&
\oplus
\big(V\wedge z\oplus V)
\ot \big(V^*\wedge z^*\oplus V^*)
&V_{2} \hbox{ and }trivial\end{array}
\]
notice
\[
\big(V\wedge z\oplus V)
\ot \big(V^*\wedge z^*\oplus V^*)
=V\ot (\C\oplus \C z)\ot 
V^*\ot (\C\oplus \C z^*)
\cong
\underbrace{V\ot V^*}_{\cong\gl(V)=\sl(V)\oplus\C \Id_V}\ot\Lambda(\C z\oplus\C z^*)
\]
\[
\cong \big(\sl(V)\oplus\C \Id_V\big)\ot\Lambda(\C z\oplus \C z^*)
\]
\[
\cong\sl(V)\ot \Lambda(\C z\oplus \C z^*)
\oplus 
\id_V\ot \Lambda(\C z\oplus \C z^*)
\]
\[
\cong
\underbrace{\sl(V)\ot\Lambda(\C z\oplus \C z^*)}
_{V_{(2)}-type}
\oplus 
\underbrace{\Lambda(\C z\oplus \C z^*)}_{trivial}
\]
From this decomposition one can get the short structure. As
a matter of example we mention that
$\J=\Lambda(\C z\oplus \C z^*)=\C\oplus \C z\oplus \C z^*
\oplus \C z^*\wedge z$, 
the associative super commutative algebra
 with $|1|=0$ and $|z|=|z^*|=1$, in particular,
$\J$ is a Jordan superalgebra. The Jordan module $\M$
is a free (associative) module over $\J$ of rank 4 ($\dim\M=16$), and $\D$ is 20-dimensional.

Since the complex is short, its homology is also short, and
\[
H ^2(\h_3,\C)\oplus H^1(\h_3,\h_3)\oplus H ^0( \h_3,\Lambda ^2\h_3)
\]
is a Lie subalgebra of the total cohomology, hence it is short.
In \cite{CG25} a basis of the cohomology is computed,
also the Lie structure is computed  explicitely. Here we give another ``explanation'' of the Lie algebra structure founded
therein.
\end{ex}

\begin{prop}
There is a nontrivial map of Lie algebras (hence an isomorphism)
\[
\sl_3(\C)\to H^2(\h_3,\C)
\oplus H^1(\h_3,\h_3)\oplus
H^0(\h_3,\Lambda^2\h_3)=H^{2,0}\oplus H^{1,1}
\oplus H^{0,2}
\]
\end{prop} 

\begin{proof}
$H^\bullet(\h_3,\Lambda \h_3)$ is a short super Lie algebra with the Poisson bracket $\{-,-\}$ (see \cite{CG25}), hence
$ H^2(\h_3,\C)
\oplus H^1(\h_3,\h_3)\oplus
H^0(\h_3,\Lambda^2\h_3)$ is also short, 
because it is a Lie subalgebra. 
If one computes the cohomology, we get that each $\sl_2$-module in terms of weight basis is given by:
\begin{align*}
    H^2(\h_3,\C) &= \langle e^2z^0,\; -e^1z^0\rangle\\
    H^1(\h_3,\h_3) &= \langle e^2\ot e_1=e,\
    e^2\ot e_2 - e^1\ot e_1=-h,\
    \; -2e^1\ot e_2=-2f\rangle \;\oplus\\
    &\quad \langle e^2\ot e_2 + e^1\ot e_1+2z^*\ot z = \id_0 \rangle\\
    &=\sl_2 \;\oplus\; \C\cdot \id_0 \\
    H^0(\h_3,\Lambda^2\h_3) &= \langle e_1z,\; e_2z\rangle
\end{align*}

\noindent Notice that it contains a copy of $\sl_2$ in $H^1(\h_3,\h_3)$. Hence, using the operations \eqref{eq: operations on J+M} with the Poisson bracket, we obtain:
\begin{align}
    (e^2\ot e_1)\cdot (e^2\ot e_1) &= -e^2\ot e_1\\
    (e^2\ot e_1)\bullet e^2z^* &= -e^2z^*\\
    (e^2\ot e_1)\bullet e_1z &= -e_1z\\
    e^2z^*\star e_1z &= -e^2\ot e_1
\end{align}
\begin{equation}\label{eq: tr_prod H(h3)-1}
    \begin{aligned}
    (e^2z^*,e_1z,e^2z^*) &= 2e^2z^*,\quad (e^2z^*,e_1z,e_1z) = -e_1z\\
    (e_1z,e^2z^*,e^2z^*) &= e^2z^*,\quad (e_1z,e^2z^*,e_1z) = -2e_1z
    \end{aligned}
\end{equation}

\noindent Define $J := \langle a\rangle$ and $M := \langle m_1,m_2\rangle$ for $a = -e^2\ot e_2$, $m_1 = e^2z^0$ and $m_2 = e_1z_0$. So, the operations $\cdot,\bullet$ and $\star$ coincide with \eqref{eq: oper on JM from sl3} and the triple product $(-,-,-)$ coincides with \eqref{eq: tr_prod on M from sl3} of Example \ref{ex-sl3}. Hence $\Tag(J,M)$ generates a copy of $\sl_3$.

\end{proof}


\nocite{*}


\end{document}